\documentclass[a4paper,12pt]{article}
\usepackage{amsmath}
\usepackage{amsthm}
\usepackage{graphicx}
\usepackage{verbatim}
\usepackage{epsfig}
\usepackage{hyperref}
\usepackage{float}
\usepackage{color}
\usepackage{fullpage}
\usepackage{amsfonts}
\usepackage{amscd}
\usepackage{amssymb}
\usepackage{graphics}
\usepackage{graphicx}
\usepackage{subcaption}
\usepackage{amsmath}
\usepackage[english]{babel}
\usepackage{bbm}
\usepackage{yfonts}
\usepackage{mathrsfs}
\usepackage{color}
\DeclareFontFamily{OML}{rsfs}{\skewchar\font'177}
\DeclareFontShape{OML}{rsfs}{m}{n}{ <5> <6> rsfs5 <7> <8> <9>
rsfs7 <10> <10.95> <12> <14.4> <17.28> <20.74> <24.88> rsfs10 }{}
\DeclareMathAlphabet{\mathfs}{OML}{rsfs}{m}{n}

\newcommand{\BE}{{\mathbb{E}}}

\newcommand{\BH}{{\mathbb{H}}}

\newcommand{\BP}{{\mathbb{P}}}

\newcommand{\BR}{{\mathbb{R}}}

\newcommand{\CA}{{\mathcal{A}}}
\newcommand{\CB}{{\mathcal{B}}}
\newcommand{\CC}{{\mathcal{C}}}

\newcommand{\CI}{{\mathcal{I}}}

\newcommand{\CL}{{\mathcal{L}}}

\newcommand{\CP}{{\mathcal{P}}}

\newcommand{\CT}{{\mathcal{T}}}

\newcommand{\CV}{{\mathcal{V}}}

\newcommand{\CZ}{{\mathcal{Z}}}

\newcommand{\dd}{{\rm d}}

\newtheorem{theorem}{Theorem}[section]
\newtheorem{proposition}[theorem]{Proposition}
\newtheorem{lemma}[theorem]{Lemma}

\newtheorem{corollary}[theorem]{Corollary}

\usepackage{enumitem, hyperref}
\makeatletter
\def\namedlabel#1#2{\begingroup
    #2%
    \def\@currentlabel{#2}%
    \phantomsection\label{#1}\endgroup
}
\makeatother

\begin{document}
\numberwithin{equation}{section} \numberwithin{figure}{section}
\title{The Poisson stick model in hyperbolic space}
\author{Erik I. Broman\footnote{Chalmers University of Technology and University of Gothenburg, email: broman@chalmers.se}, Johan H. Tykesson\footnote{Chalmers University of Technology and University of Gothenburg, email: johant@chalmers.se}}
\maketitle

\begin{abstract}
In this paper we study the Poisson stick model in two dimensional hyperbolic space 
$\BH^2,$ where the sticks all have length $L.$
Typically, percolation models in hyperbolic space 
undergo two phase transitions as the intensity $\lambda$ varies, namely the percolation phase 
transition and the uniqueness phase transition. For the Poisson stick model, 
the critical intensities at which these transitions
occur will depend on $L$, and in this paper we study the asymptotic behavior of these critical points 
as $L\to \infty.$ Our main results show that the critical point for the percolation phase transition 
scales like $L^{-2},$ while the critical point for the uniqueness phase transition 
scales like $L^{-1}.$ Comparing these results to the analogous results in Euclidean space show that 
the behavior of the percolation phase transition is the same in these two settings, while 
the uniqueness phase transition scales differently. 
\end{abstract}


\section{Introduction}
The Poisson stick model was first studied in \cite{R_91}, although there is some earlier work
such as \cite{DK_84} and \cite{H_85} where it was implicitly included as part of a more general setup. 
The stick model has 
been used to study various phenomena in material sciences (see for instance \cite{MNT_14}, \cite{SKKSW_03}
and \cite{TE_19} and the references therein). In fact, most of the papers concerning stick models 
are in the applied sciences
and not in the mathematical literature. Recently however, the paper \cite{B_22} studied the Poisson stick 
model in Euclidean space, and in particular the behavior of a critical parameter was studied 
as the length of the sticks grows. The exact result is stated below. The purpose of this paper is to 
investigate how the underlying geometry affects those results, and in particular, we here study 
the case when the underlying geometry is hyperbolic rather than Euclidean.

The Poisson stick model that we study here is a Poisson process $\omega^{\lambda,L}$ of sticks 
$l_L$ of length $L,$ and this model
is invariant in distribution under isometries of $\BH^2$ (see Section \ref{sec:Poistick} for details).
Here, $\lambda>0$ denotes the intensity of this process and as $\lambda$ varies, the model 
undergoes two phase transitions as we now explain. First, we let $\CC=\CC(\omega^{\lambda,L})$ be defined 
by
\begin{equation} \label{eqn:CCdef}
\CC:=\bigcup_{l_L\in \omega^{\lambda,L}} l_L
\end{equation}
so that $\CC$ is the occupied set, and then we define 
\begin{equation} \label{eqn:deflambdac}
\lambda_c(L)
:=\inf\{\lambda>0: \BP(\CC(\omega^{\lambda,L}) 
\textrm{ contains an unbounded connected component})>0\}
\end{equation}
and also 
\begin{equation} \label{eqn:deflambdau}
\lambda_u(L) 
:=\inf\{\lambda>0: \BP(\CC(\omega^{\lambda,L}) 
\textrm{ contains a unique unbounded connected component})>0\}. 
\end{equation}
We note that it is not hard to show that the probabilities involved in these definitions must be either 
0 or 1, see further Section \ref{sec:Poistick}.

Our first result is the following and it concerns $\lambda_c(L).$
\begin{theorem} \label{thm:lambdacmain}
For the Poisson stick model in $\BH^2,$ we have that 
for every $0<L<\infty$ large enough,
\[
\frac{\pi}{2}L^{-2} \leq  \lambda_c(L) \leq \frac{32 \pi}{\sqrt{3}-1} L^{-2}.
\]
\end{theorem}
Our second result concerns the uniqueness phase transition.
\begin{theorem} \label{thm:lambdaumain}
For the Poisson stick model in $\BH^2,$ we have that 
for every $0<L<\infty$ large enough,
\[
\frac{\pi}{2} L^{-1} \leq  \lambda_u(L) \leq 5\sqrt{2} \pi L^{-1}.
\]
\end{theorem}
\noindent
{\bf Remark:}  We remark that it is possible to improve both of 
the constants in the upper 
bounds by making more precise calculations in the proofs. However, there is no clear
way to improve the lower bounds by adjusting the existing proofs. As explained in the 
remarks after the proof of Proposition \ref{prop:half-plane} and at the very end of 
the paper, we choose relative simplicity in place of obtaining the best possible
but still non-sharp value on the constants in the upper bounds.

\medskip

Although we believe that Theorems \ref{thm:lambdacmain} and \ref{thm:lambdaumain} have 
intrinsic values, we want to contrast these results to what is known in Euclidean space
(see \cite{B_22}). To that end, let $\lambda^E_c(L)$ and $\lambda^E_u(L)$ denote the 
critical values for the Poisson stick process in $\BR^d,$ defined analogously to 
\eqref{eqn:deflambdac} and \eqref{eqn:deflambdau}
respectively. The main result of \cite{B_22} 
was that 
for any $d\geq 2,$ $\lambda^E_c(L)\sim L^{-2}$ as $L\to \infty.$ More precisely, for any $ d \geq 2,$ 
there exist two constants $0<c^E(d)<C^E(d)<\infty$ such that for any $L$ large enough 
\[
c^E(d)L^{-2}\leq \lambda^E_c(L) \leq C^E(d) L^{-2}.
\]
As a side-note, the constants $c^E(d)$ and $C^E(d)$ were provided explicitly, but  
were very far apart and presumably neither were close to being sharp.
The critical value $\lambda^E_u(L)$ was not mentioned in \cite{B_22}. However, we conjecture
that in the Euclidean case,  whenever the stick model there percolates, there is in fact a 
unique connected component. This has not been formally proven, but the arguments 
in \cite{MR_96} (which in turn builds on classical percolation arguments, see for 
instance \cite{G_99}) transfers to this current setting so that we are very comfortable 
making this conjecture. Thus, in the Euclidean case, we conjecture that in fact 
also $\lambda^E_u(L)\sim L^{-2}$ since we believe that $\lambda^E_c(L)=\lambda^E_u(L)$. As we see, 
our main results show that for $\BH^2,$ the scaling of $\lambda_c(L)$  is the same as 
$\lambda^E_c(L)$ when $d\geq 2,$ but that the scaling of $\lambda_u(L)$ is different from 
$\lambda^E_u(L)$ (if indeed the conjecture about $\lambda^E_u(L)$ holds). 
There is an intuitive reason for this difference, but the discussion is too long
for an introduction. We therefore postpone explaining the intuition until the end of Section 
\ref{sec:preliminaries}.

In this paper, we only consider $\BH^2$ rather than $\BH^d$ for $d\geq 2.$
Of course, in $\BH^d$ where $d\geq 3,$ one would need to consider sticks of width 
say 1. As we argue in the remark at the end of Section \ref{sec:lambdac}, one should be
able to adjust the proof of Theorem \ref{thm:lambdacmain} to
$\BH^d$ for any $d \geq 3$ (with an analogous result). However, we anticipate that there could 
be considerable added technical difficulties. Therefore,
we are afraid that generalizing the arguments to arbitrary $d\geq 2$ would obfuscate the ideas in
technicalities, without providing much added value. Indeed, the similarities between 
hyperbolic and Euclidean spaces for $\lambda_c$,  and the difference in scaling behavior 
for $\lambda_u$ are demonstrated already for $d=2.$

The proof of Theorem \ref{thm:lambdaumain} uses planarity in a crucial way for the upper 
direction. In addition, the proof of the lower bound relies on a result which has only been 
proven for the planar case.

We also note that, as stated in \cite{B_22}, the result concerning $\lambda_c^E(L)$
for $\BR^2$ is very easy to prove 
through a simple scaling argument. Indeed, consider the model with parameters $(\lambda,L)$ 
and let $a>0$. Applying the map $x\mapsto a x$ to the model results in a model with 
parameters $(\lambda/a^2,a L)$. From this, one sees that $\lambda^E_c(a L)=a^{-2}\lambda^E_c(L)$ from 
which it follows that $\lambda^E_c(L)=\lambda^E_c(1)L^{-2}$. Therefore, the results in Euclidean space 
are really only interesting for dimensions $d \geq 3.$ Furthermore, in order to obtain sticks that 
can intersect, the sticks where there taken to have width 1. 

A scaling argument similar to the one described for the model in $\BR^2$ above 
does not work 
in the hyperbolic plane. Indeed, for $a>0$ consider the map 
\[
T_a: (\rho(x),\theta(x))\in \BH^2 \to (\rho(x)/a,\theta(x))\in \BH^2
\]
where $(\rho(x),\theta(x))$ denotes $x$ in hyperbolic polar coordinates (see 
also Section \ref{sec:2dhyp}). If we apply this map to the Poisson process 
$\omega^{\lambda,L}$ we will not get a Poisson process which is invariant 
under the isometries 
of ${\mathbb H}^2$. In fact, if $\omega^{\lambda}$ denotes the Poisson point process corresponding 
to the center-points of sticks in $\omega^{\lambda,L}$, it is straightforward to show that the process $T_a(\omega^{\lambda})$ will 
still be a Poisson point process, but not homogeneous. 

We end this section by presenting an outline of the rest of the paper. In Section \ref{sec:modelsanddef},
we give some background on hyperbolic space, and we take more care in defining the model and 
setting notation. 
Section \ref{sec:preliminaries} establishes some preliminary results, while Sections 
\ref{sec:lambdac} and \ref{sec:lambdau} are devoted to the proofs of Theorems \ref{thm:lambdacmain}
and \ref{thm:lambdaumain} respectively.

\section{Models and definitions} \label{sec:modelsanddef}
As we have already seen, we will distinguish Euclidean objects (such as the 
Euclidean metric or $\lambda_c^E$ etc) by using ``$E$'' somewhere in the notation. 
For hyperbolic space, we instead use ``$h$''.

\subsection{$2$-dimensional hyperbolic space} \label{sec:2dhyp}
We let $\BH^2$ denote $2$-dimensional hyperbolic space, represented throughout the paper 
by the Poincar\'e disc model. In this model we consider the unit disc 
$\{x\in \BR^2: d_E(o,x)<1\}$ 
(where $d_E$ denotes $2$-dimensional Euclidean metric) in Euclidean space,
equipped with the hyperbolic metric 
\[
d_h(x,y)=\cosh^{-1}\left(1+2\frac{d_E(x,y)^2}{(1-d_E(o,x)^2)(1-d_E(o,y)^2)}\right).
\]

It will often be useful to represent a point $x\in \BH^2$ in polar coordinates. 
We will then write $x\in \BH^2$ as $(\rho(x),\theta(x))$ where $\rho=\rho(x)=d_h(o,x)$ is 
the hyperbolic distance between the point $x\in \BH^2$ and the origin, 
and the angle $\theta(x)\in [0, 2 \pi)$ is defined 
counter-clockwise from the positive horizontal axis.
On $\BH^2$ we consider the (hyperbolic) area measure $v^h.$ It is known 
(see \cite{S_04}  Chapter 17 equation 17.47 and the preceding pages) that for a set $A\subset \BH^2,$ 
the area $v^h(A)$ can be written as
\begin{equation} \label{eqn:volumemeasureH}
   v^h(A)=\int_A \sinh(\rho) \dd\rho \dd\theta, 
\end{equation}
where $\dd \rho$ is Lebesgue measure on $[0,\infty)$ and $ \dd \theta$ denotes the solid angle element 
which here simplifies (since $d=2$) to Lebesgue measure on $[0,2\pi).$
Throughout, $B^h(x,\rho)=\{y\in \BH^2: d_h(x,y)\leq \rho\}$ will denote a closed 
 ball in 
$\BH^2$ centered at $x$ and with radius $\rho>0.$ We note for future reference that 
\begin{equation} \label{eqn:volhypball}
    v^h(B^h(x,\rho))=2 \pi \int_0^\rho \sinh(r)dr=2 \pi (\cosh(\rho)-1)
    =4\pi \sinh^2(\rho/2).
\end{equation}

We will have use of three hyperbolic trigonometric laws.  
Consider therefore a triangle
with side lengths $A,B,C$ and opposing angles $\alpha,\beta,\gamma.$ The first rule 
is the hyperbolic law of sines (see Chapter 7.12 
of \cite{B_93}) which states that 
\begin{equation} \label{eqn:hypsinerule}
\frac{\sinh(A)}{\sin \alpha}=\frac{\sinh(B)}{\sin \beta}=\frac{\sinh(C)}{\sin \gamma}.
\end{equation}
The second is a hyperbolic law of cosines (see again Chapter 7.12 
of \cite{B_93})
which states that 
\begin{equation} \label{eqn:hypcosinerule}
\cos \alpha= \sin \beta \sin \gamma \cosh(A)-\cos \beta \cos \gamma.   
\end{equation}
The third is another hyperbolic law of cosines (see Chapter 7.10 of \cite{B_93})
which holds in the case that $A,B$ are both infinite
so that $\gamma=0,$ and it states that
\begin{equation}\label{eqn:hypcosinerule_infinite}
\cosh(C)=\frac{1+\cos(\alpha)\cos(\beta)}{\sin(\alpha)\sin(\beta)}.    
\end{equation}

\subsection{Geodesics and Isometries in $\BH^2$}
Given an angle $\theta\in [0,\pi),$ we let $g_\theta$ denote the geodesic 
\begin{equation} \label{eqn:defgeodesic}
  g_\theta:=\{x\in \BH^2: \theta(x)=\theta, \rho(x)<\infty\}
\cup
\{x\in \BH^2: \theta(x)=\theta+\pi, \rho(x)<\infty\}.  
\end{equation}
Since for $x\in \BH^2$ we have that $\theta(x)\in[0,2\pi)$, we will write
$g_{\theta(x) \pmod \pi}$ 
for the geodesic which goes through $x$ and the origin $o.$ We will let $g_1,g_2$ etc denote 
geodesics which do not necessarily pass through $o.$ Of course, there is some clash of notation 
as $g_1$ could be interpreted as $g_{\theta(x)=1},$ but this will be clear from context.

It is well known (see Chapter 7.4 of \cite{B_93}) that the set of isometries on $\BH^2$ are the maps
\[
z \to \frac{az+\bar{c}}{cz+\bar{a}}, \ \ z \to \frac{a\bar{z}+\bar{c}}{c\bar{z}+\bar{a}},
\]
where $|a|^2-|c|^2=1.$ The first set of isometries consists of Möbius transformations,  while the second
set consists of maps which are the result of taking the conjugate $\bar{z},$ and then applying a 
Möbius transformation.
It is also well known that all of these maps are conformal maps and therefore locally 
preserve angles.

We will let $\CI^x$ denote the unique orientation-preserving 
isometry which maps $o$ to $x$ and which leaves 
the geodesic $g_{\theta(x) \pmod \pi}$ invariant, and in addition 
fixes the endpoints on $\partial \BH^2$
of $g_{\theta(x) \pmod \pi}.$ The isometry $\CI^x$  is 
sometimes referred to as a ``translation''.

\subsection{The Poisson stick process} \label{sec:Poistick}
Throughout, we shall let $l_L(x,\phi)$ denote a line segment of length $L$,
centered at $x,$ and where the angle between $g_{\theta(x) \pmod \pi}$ and
$l_L(x,\phi),$ measured counterclockwise starting from $g_{\theta(x) \pmod \pi},$ 
is $\phi\in[0,\pi).$ 
That is, the angle is defined by using the geodesic 
$g_{\theta(x) \pmod \pi}$ passing through $o$ and $x$ as the ``base line''.
More formally, for $\phi\in[0,\pi)$, the line segment $l_L(o,\phi)$ is defined as 
\[
l_L(o,\phi):=\{y\in \BH^2: \theta(y)=\phi, \rho(y)\leq L/2\}
\cup 
\{y\in \BH^2: \theta(y)=\phi+\pi, \rho(y)\leq L/2\}.
\]
It is clearly the case that $l_L(o,\phi) \subset g_\phi$ and that the angle between 
the horizontal line and $l_L(o,\phi)$ is $\phi.$
Next, we define 
\begin{align*}
l_L(x,\phi)& =\CI^x\left(l_L(o,\phi+\theta(x) \pmod \pi)\right) \\
& =\{y\in \BH^2:y=\CI^x(z) \textrm{ for some } z\in l_L(o,\phi+\theta(x) \pmod \pi)\}.
\end{align*}
Informally, we can think of this as first rotating $l_L(o,\phi)$ so that the angle between 
$l_L(o,\phi+\theta(x) \pmod \pi)$ and $g_{\theta(x) \pmod \pi}$ is $\phi,$ and secondly 
we translate $l_L(o,\phi+\theta(x) \pmod \pi)$ along $g_{\theta(x) \pmod \pi}$ so that its 
center becomes $x$. Since by the preceding subsection, the isometry $\CI^x$ 
preserves angles, we see that the angle between $l_L(x,\phi)$ and $g_{\theta(x) \pmod \pi}$
is indeed $\phi.$
Furthermore, since $l_L(o,\phi+\theta(x) \pmod \pi)$ is a subset of the geodesic 
$g_{\phi+\theta(x) \pmod \pi}$, $l_L(x,\phi)$ is also part of a geodesic and 
is therefore in fact a line segment. 
We choose to define 
$l_L(x,\phi)$ in this way out of later convenience.

We will let $\CL^L:=\{l_L(x,\phi): x\in \BH^2, \phi\in[0,\pi)\}$ denote the 
set of all 
line segments in $\BH^2$ of length $L.$ Furthermore, for any $A \subset \BH^2$, we define  
\begin{equation} \label{eqn:defCLA}
\CL^L(A)=\{l_L(x,\theta)\in \CL^L: l_L(x,\theta)\cap A \neq \emptyset\}.    
\end{equation}

Recall the hyperbolic area measure $v^h$ as expressed in \eqref{eqn:volumemeasureH}, and consider the intensity measure $\mu$ defined through 
\begin{equation} \label{eqn:defmud}
    \dd \mu=\dd v^h \otimes \dd\Phi 
\end{equation}
where $\dd \Phi$ denotes uniform probability measure on $[0,\pi).$
We then let $\omega^{\lambda,L}$ be the Poisson point process on the space 
$\BH^2 \times [0,\pi)$ using $\lambda \mu$ 
(where $\lambda>0$ is a parameter) as the intensity 
measure. To any point $(x,\phi)\in \omega^{\lambda,L}$ we identify the line 
segment $l_L(x,\phi).$
Because of this identification, we will sometimes abuse notation somewhat and simply write
$l_L(x,\phi) \in \omega^{\lambda,L}.$ 
For any $\CA \subset \CL^L$ and isometry $\CI,$ we let 
\[
\CI(\CA)=\{\CI(l_L(x,\phi)): l_L(x,\phi)\in \CA\} \subset \CL^L,
\]
where the inclusion holds since $\CI$ is an isometry and therefore preserves distances. 
We have that 
\[
\mu(\CI(\CA))=\mu(\CA),    
\]
which is an easy consequence of the definition \eqref{eqn:defmud}, since $v^h$ is clearly 
invariant under isometries and $\dd \Phi$ is uniform measure on $[0,\pi).$

From now on we will refer to a line segment $l_L(x,\phi)$ as a ``stick'' of length $L,$
and if we talk about a generic stick we shall simply write $l_L.$  Furthermore, 
we will write $l[x,y]$ for a line segment (not necessarily of length $L$) 
with endpoints $x,y\in \BH^2 \cup \partial \BH^2.$ We remark that although a stick 
$l_L(x,\phi)$ is a line segment, it will be convenient to distinguish generic line segments
from sticks that can be part of the stick process $\omega^{\lambda,L}$.

\subsection{The stick process and phase transitions.} \label{sec:stickprocess}
Recall the definition of $\CC=\CC(\omega^{\lambda,L})$ in \eqref{eqn:CCdef} and the definitions
of $\lambda_c(L)$ and $\lambda_u(L)$ in \eqref{eqn:deflambdac}
and \eqref{eqn:deflambdau} respectively.
It was proven in Proposition 2.1 of \cite{BT_15} (see also the references given there) that in a similar context,
any event which is invariant
under isometries, must have probability either 0 or 1. The proof of this fact is
straightforward and based on standard methods, and so we do not repeat 
the argument here. Therefore, for any $L,$ we have that
\[
\BP(\CC(\omega^{\lambda,L}) 
\textrm{ contains an unbounded connected component})\in \{0,1\}.
\]
Furthermore, it is clear that this probability 
is non-decreasing in $\lambda,$ and so we conclude that in fact 
\begin{align*}
    \lambda_c(L) & =\inf\{\lambda>0: \BP(\CC(\omega^{\lambda,L}) 
\textrm{ contains an unbounded connected component})=1)\} \\
& =\sup\{\lambda>0: \BP(\CC(\omega^{\lambda,L}) 
\textrm{ contains an unbounded connected component})=0)\}.
\end{align*}
Furthermore, in the same way,
\[
\BP(\CC(\omega^{\lambda,L}) 
\textrm{ contains a unique unbounded connected component})\in \{0,1\}.
\]
It is not immediate that this probability
is non-decreasing in $\lambda.$ Informally, when $\lambda$ increases, we are adding 
more line segments, and these could potentially form a new, distinct, 
unbounded connected component. However, it is possible to prove that this 
probability is in fact non-decreasing in $\lambda$.
For proofs in very similar circumstances, see Section 5 of \cite{HJ_06} or 
Lemmas 5.4 and 5.5 of \cite{T_07}.

\section{Preliminary results} \label{sec:preliminaries}
The purpose of this section is to establish a few preliminary results which shall be used in later sections 
in order to prove our main results, i.e.\ Theorems \ref{thm:lambdacmain} and \ref{thm:lambdaumain}.

First we consider two points $x,y\in \partial B^h(o,\rho),$ and we let $\varphi\in[0,\pi)$ be the angle 
between them. The following lemma establishes the hyperbolic distance between $x$ and $y.$
\begin{lemma} \label{lemma:xydist}
With $x,y$ and $\varphi$ as above, we have that
\[
d_h(x,y)=2\sinh^{-1}\left(\sinh(\rho) \sin(\varphi/2)\right).
\] 
\end{lemma}
\begin{proof}
Consider the line segment $l[x,y]$ connecting $x$ and $y,$ 
and let $g=g_{(\theta(x)+\theta(y))/2}$ be the geodesic
bisecting the triangle with corners $o,x,y.$ It is clear that $l[x,y]$ and $g$ intersects orthogonally. 
We next consider the triangle with corners $o,x$ and $l[x,y]\cap g.$ The hyperbolic law of sines 
\eqref{eqn:hypsinerule}
tells us that 
\[
\frac{\sinh(\rho)}{\sin(\pi/2)}=\frac{\sinh(d_h(x,y)/2)}{\sin(\varphi/2)}
\textrm{ so that }
\sinh(d_h(x,y)/2)=\sinh(\rho) \sin(\varphi/2),
\]
from which it follows that 
\[
d_h(x,y)=2\sinh^{-1}\left(\sinh(\rho) \sin(\varphi/2)\right)
\]
as claimed.
\end{proof}

We now consider a set $\CP_N$ consisting of $N$ points $x_0,\ldots, x_{N-1}$
placed equidistantly on $\partial B^h(o,\rho).$ We assume that $\theta(x_0)=0,$
and that all points are placed around the circle $\partial B^h(o,\rho)$ in order
so that $d_h(x_k,x_{k+1})=d_h(x_{N-1},x_0)$ for every $k=0,\ldots,N-2.$ 
\begin{lemma} \label{lemma:ballsoncircle}
If $N=\lceil \pi \sinh(\rho) \rceil,$ then, 
\[
\partial B^h(o,\rho) \subset \bigcup_{k=0}^{N-1} B^h(x_k,1). 
\]
Furthermore, whenever $\rho \geq 3,$ we have that for any $k=0,\ldots,N-3$
\begin{equation} \label{eqn:pointsoncircledist}
  d_h(x_k,x_{k+2})=d_h(x_{N-2},x_0)=d_h(x_{N-1},x_1) >\frac{5}{2}.   
\end{equation}
\end{lemma}
\begin{proof}
The circle $\partial B^h(o,\rho)$ has hyperbolic length
$2 \pi \sinh(\rho).$ Therefore, since there are $N=\lceil \pi \sinh(\rho) \rceil$ points, 
the angle $\varphi$ between two consecutive points in $\CP_N$ becomes 
\begin{equation} \label{eqn:varphibounds}
 \varphi=\frac{2 \pi}{\lceil \pi \sinh(\rho) \rceil}
\in \left[ \frac{2}{\sinh(\rho)+1},\frac{2}{\sinh(\rho)}\right].   
\end{equation}
Consider now $x_0,x_1$ and let $x_{1/2} \in \partial B^h(o,\rho)$ be such that 
$\theta(x_{1/2})=\varphi/2$
so that $x_{1/2}$ is the halfway point between $x_0$ and $x_1$ on $\partial B^h(o,\rho).$ 
Using the elementary inequality $\sin x \leq x$ we then conclude from Lemma \ref{lemma:xydist}
that 
\begin{equation} \label{eqn:phibounds}
   d_h(x_0,x_{1/2}) =2\sinh^{-1}\left(\sinh(\rho) \sin(\varphi/4)\right)
\leq 2\sinh^{-1}\left(\sinh(\rho) \varphi/4\right)  
\leq 2\sinh^{-1}\left(1/2\right)<1  
\end{equation}
where we used \eqref{eqn:varphibounds} in the second inequality, and where
the last inequality can be verified by elementary means. This shows that 
$x_{1/2}\in B^h(x_0,1)$ (and also that $x_{1/2}\in B^h(x_1,1)$) and so we conclude that 
any point on $x\in \partial B^h(o,\rho)$ must be such that $d_h(x,\CP_N)<1.$ This 
proves the first statement of the lemma.

For the second statement, we use the elementary inequality $\sin(x)\geq \frac{9}{10}x$ whenever 
$0\leq x\leq 1/2.$ Therefore, if $\rho \geq 3$ so that 
\[
\varphi\leq \frac{2}{\sinh(3)}<\frac{1}{2},
\]
we can use \eqref{eqn:phibounds} to see that for $x_0,x_{2},$ 
\begin{align*}
d_h(x_0,x_{2}) & =2\sinh^{-1}\left(\sinh(\rho) \sin(\varphi)\right)
\geq 2\sinh^{-1}\left(\sinh(\rho) \frac{9}{10} \varphi \right) \\
& \geq 2\sinh^{-1}\left(\sinh(\rho) \frac{9}{10} \cdot \frac{2}{\sinh(\rho)+1}\right) 
=2\sinh^{-1}\left(\frac{18}{10(1+1/\sinh(\rho))}\right) \\
& \geq 2\sinh^{-1}\left(\frac{18}{10(1+1/\sinh(3))}\right)>\frac{5}{2},
\end{align*}
where again, the last inequality can be verified by elementary means.
This proves the second statement of the lemma.
\end{proof}

\medskip

The last result of this section (Lemma \ref{lemma:altPoisson}) 
concerns the restriction of the Poisson point process
$\omega^{\lambda,L}$ to those sticks $l_L(x,\phi)$ which intersect the positive horizontal line, 
i.e.\ the ray $g_0^+:=\{x\in \BH^2: 0\leq \rho(x)<\infty, \theta(x)=0\}.$ 
We let (as before)
\[
\CL^L(g_o^+):=\{l_L(x,\phi)\in \CL^L: l_L(x,\phi) \cap g_0^+ \neq \emptyset\}
\]
and 
\[
\omega^{\lambda,L}_{|g_o^+}:=\{l_L(x,\phi)\in \omega^{\lambda,L}: l_L(x,\phi) \cap g_0^+ \neq \emptyset\}.
\]
By the restriction theorem (Theorem 5.2 of \cite{LP_18}) $\omega^{\lambda,L}_{|g_o^+}$ is a Poisson 
process with intensity measure $\mu_{|g_o^+}$ defined by 
\[
\mu_{|g_o^+}(A)=\mu(A \cap \CL^L(g_o^+))
\]
for every (measurable) $A \subset \CL^L.$
By understanding the process $\omega^{\lambda,L}_{|g_o^+}$ we will, for example, be able to 
easily calculate the expected number of sticks which hit a certain 
line segment within a certain angle interval. This will be useful on several occasions in the coming 
sections.
The informal interpretation of Lemma \ref{lemma:altPoisson} is that the 
restricted Poisson process $\omega^{\lambda,L}_{|g_o^+}$, can be generated
through the following three steps (see also Figure \ref{fig:restrict}).

\begin{figure}[h]
\vspace*{-10mm}
\centering
\begin{subfigure}{\textwidth}
\centering
\includegraphics[page=1, clip=true, trim=70 150 270 70 , width=12cm]{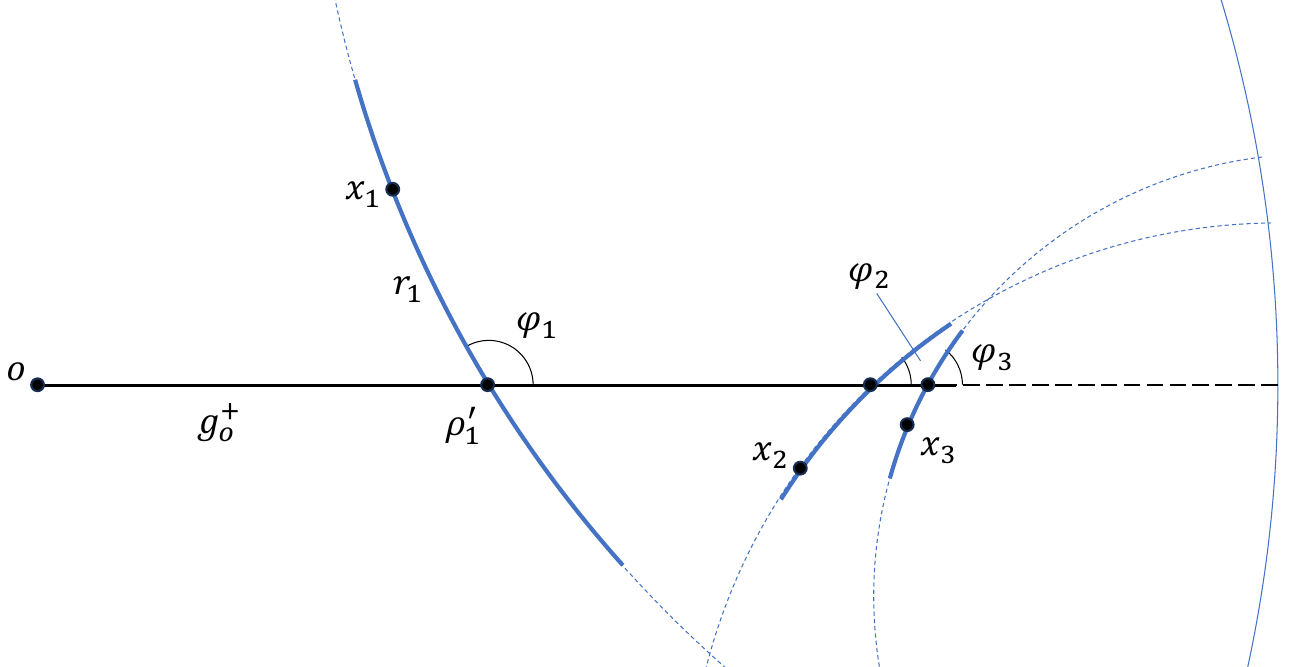} 
\caption{} \label{fig:restrictA}
\end{subfigure}

\begin{subfigure}{\textwidth}
\centering
\includegraphics[page=2, clip=true, trim=70 150 270 70 , width=12cm]{Restrict.pdf} 
\caption{}\label{fig:restrictB}
\end{subfigure}

\caption{Figure \ref{fig:restrictA} shows $g_0^+$ and the first three intersection points of sticks from 
$\omega^{\lambda,L}$ which intersect it. The distance to the first intersection point is 
$\rho'_1$ and this is marked along with 
the three angles $\varphi_1,\varphi_2$ and $\varphi_3$. The solid blue line segments in Figure
\ref{fig:restrictA} are the sets of points 
along the geodesics (dashed) which are at distance at most $L/2$ from the intersection points. 
The centerpoints $x_1,x_2$ and $x_3$ are chosen uniformly along these blue line segments.
Figure \ref{fig:restrictB} shows the resulting sticks in red.} \label{fig:restrict}

\end{figure}

\begin{enumerate}
\item First, we take a homogeneous (using the hyperbolic metric $d_h$) Poisson 
point process on $g_0^+$ with intensity $\frac{2}{\pi} \lambda L.$

\item Secondly, at every point $(\rho',0)\in g_0^+$ 
of this point process, 
we place a stick at angle $\varphi\in [0,\pi)$ to $g_0^+$ according to the probability
measure $\frac{1}{2}\sin \varphi \dd \varphi.$ This is done independently for every stick.

\item Thirdly, we determine the location of the center $x\in \BH^2$ of the stick by 
letting the distance $r$ between $x$ and the intersection point 
$(\rho',0)\in g_0^+$ 
of the stick with $g_0^+,$ be picked uniformly in $[-L/2,L/2].$ Again, 
this is done independently for every stick.
\end{enumerate}

We will prove Lemma \ref{lemma:altPoisson} using the Mapping Theorem (see Theorem 5.1 of \cite{LP_18}).
Recall that a stick $l_L(x,\phi)$ is, in 
hyperbolic polar coordinates, determined by the triple $(\rho(x),\theta(x),\phi).$ 
If a stick $l_L(x,\phi)$ intersects $g_0^+,$ then this stick can alternatively 
be represented by the triple 
$(\rho',\varphi,r)$ (using the notation above, see also Figures \ref{fig:restrict} 
and \ref{fig:triangles}). 
\begin{figure}[h] 
\begin{center}
\includegraphics[page=1, clip=true, trim=85 60 600 220 , width=8cm]{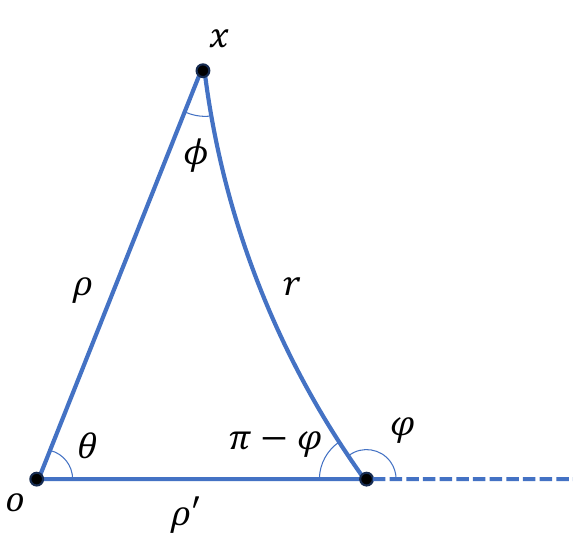} 
\caption{The triangle $\CT$.} \label{fig:triangles}   
\end{center}
\end{figure}
For a stick $l_L(x,\phi)\in \CL^L(g_0^+)$, let $T$ denote the mapping which maps the triple 
$(\rho(x),\theta(x),\phi)$ to the triple 
$(\rho',\varphi,r).$ Clearly, this mapping is 1-to-1 from 
$\{(\rho(x),\theta(x),\phi):l_L(x,\phi)\in \CL(g_0^+)\}$ to 
$\{(\rho',\varphi,r):(\rho',\varphi,r)\in [0,\infty) \times [0,\pi) \times [-L/2,L/2]\}.$ 
According to Theorem 5.1 of \cite{LP_18}, $T\left(\omega_{g_0^+}^{\lambda,L}\right)$ 
is a Poisson point process
on the space $[0,\infty) \times [0,\pi) \times [-L/2,L/2] $ with intensity 
measure $\mu'=T(\mu)$ defined by the relation 
\begin{equation} \label{eqn:muprimemurelation}
  \mu'(\CA)=T(\mu)(\CA)=\mu(T^{-1}(\CA))  
\end{equation}
for every measurable event $\CA.$ Our following result describes the measure $\mu'$ in terms of 
$(\rho',\varphi,r).$ We also note that this result can be extended to include the rest of the 
geodesic $g_0.$ Here and below, we let $I(\cdot)$ denote an indicator function.

\begin{lemma} \label{lemma:altPoisson}
The measure $\mu'$ is given by 
\[
\dd \mu'= \frac{1}{\pi}I(\rho'>0)\dd \rho' 
\otimes I(\varphi\in[0,\pi))\sin \varphi \dd \varphi
\otimes I(r\in[-L/2,L/2])\dd r.
\]
\end{lemma}
\begin{proof}
By standard methods, it suffices to show that for any event of the form 
\[
\CA=[a,b] \times [\varphi_1,\varphi_2]\times [r_1,r_2] ,
\]
where $0\leq a<b<\infty,\ 0\leq \varphi_1<\varphi_2<\pi$ and $-L/2<r_1<r_2<L/2,$ 
we have that 
\begin{equation} \label{eqn:muprimeA_org} 
\mu'(\CA)=\frac{1}{\pi}
\int_{a}^b \int_{\varphi_1}^{\varphi_2} \int_{r_1}^{r_2} 
\sin \varphi  \dd r \dd \varphi  \dd \rho'.    
\end{equation}
Furthermore, by invariance of the measure $\mu,$ we see that for $0<r_1<r_2<L/2$ we must have that
\[
\mu(T^{-1}([a,b] \times [\varphi_1,\varphi_2]\times [r_1,r_2] ))
=\mu(T^{-1}([a,b] \times [\varphi_1,\varphi_2]\times [-r_1,-r_2] )),
\]
and therefore, it suffices to show \eqref{eqn:muprimeA_org} for $0<r_1<r_2<L/2$
(note that the contribution of $r_1=0$ to \eqref{eqn:muprimeA_org} must be 0).
Furthermore, if we let $\epsilon>0$ be such that $(b-a)/\epsilon$ is an integer,
then we have that 
\begin{align*}
& \mu(T^{-1}([a,b] \times [\varphi_1,\varphi_2]\times [r_1,r_2] )) \\
& =\sum_{k=0}^{\frac{b-a}{\epsilon}-1}  
\mu(T^{-1}([k\epsilon,(k+1)\epsilon) \times [\varphi_1,\varphi_2]\times [r_1,r_2] )) \\
& =\frac{b-a}{\epsilon}\mu(T^{-1}([0,\epsilon) 
\times [\varphi_1,\varphi_2] \times [r_1,r_2] )), 
\end{align*}
by using that the contribution of a single point ($\rho'=b$) is 0 and the invariance of $\mu.$

Consider now $\CA$ as above with $r_1>0.$
Using \eqref{eqn:muprimemurelation}, the definition of 
$\mu$ (i.e.\ \eqref{eqn:defmud}) and \eqref{eqn:volumemeasureH},
we see that 
\begin{align} \label{eqn:muprimeA}
    & \mu'(\CA)=\mu(T^{-1}(\CA)) 
    =\frac{b-a}{\epsilon}\mu(T^{-1}([0,\epsilon) \times [\varphi_1,\varphi_2]\times [r_1,r_2] )),\\
& =\frac{b-a}{\epsilon }\int_{\rho=0}^\infty \int_{\theta=0}^{2 \pi} \int_{\phi=0}^{\pi} 
I((\rho,\theta,\phi)\in T^{-1}([0,\epsilon) \times [\varphi_1,\varphi_2]\times [r_1,r_2] )) 
\frac{1}{\pi}\dd \phi \dd \theta \sinh(\rho) \dd \rho, 
\nonumber 
\end{align}
since $\dd \Phi=\frac{1}{\pi}I(0\leq \phi< \pi)\dd \phi.$
It is tempting to attempt to 
show that \eqref{eqn:muprimeA} equals \eqref{eqn:muprimeA_org} by a change of 
variables and calculating the Jacobian. However, this turns out to be rather 
involved and instead we opt for a more direct approach which will be simplified by 
taking $\epsilon>0$ small.
Indeed, we will proceed by estimating the right hand side of \eqref{eqn:muprimeA},
and then we will take the limit as $\epsilon \to 0.$ We therefore need to 
understand what the condition
\[
(\rho',\varphi,r)\in [0,\epsilon)\times [\varphi_1,\varphi_2]\times [r_1,r_2]
\]
implies for $(\rho,\theta,\phi)= T^{-1}(\rho',\varphi,r)$. 
For example, the condition 
that $\rho'=\rho'(\rho,\theta,\phi)\in[0,\epsilon)$ will mean that the integrand 
of \eqref{eqn:muprimeA} is 0 for certain values of $(\rho,\theta,\phi).$
In order to understand $T^{-1}(\CA)$, consider a fixed $r_1>0$ and 
the triangle $\CT$ defined by the points $o,(\rho',0)$ and $(\rho,\theta),$
see Figure \ref{fig:triangles}. 
We will also assume that $\epsilon>0$ is much smaller 
that $r_1>0.$ 

For $\rho,$ we note that by the triangle inequality 
we have that $\rho-\rho' \leq r \leq \rho +\rho'.$ Furthermore, $\rho'<\epsilon$ 
by assumption which implies that $\rho-\epsilon \leq r \leq \rho +\epsilon.$ Thus
\begin{equation} \label{eqn:rhoimpliedbounds}
(\rho',\varphi,r)\in[0,\epsilon) \times [\varphi_1,\varphi_2]
\times [r_1,r_2] 
\Rightarrow \rho\in[r_1-\epsilon,r_2+\epsilon] .
\end{equation}

Next, we establish similar bounds for $\theta.$ Recall the notation 
$v^h$ from \eqref{eqn:volumemeasureH}.
It is well known (see \cite{B_93} p 150) that 
the area of the hyperbolic 
triangle $\CT$ is given by the expression (with notation as in Figure \ref{fig:triangles})
\begin{equation}\label{eqn:triarea}
v^h(\CT)=\pi-(\theta+(\pi-\varphi)+\phi)=\varphi-\theta-\phi.  
\end{equation}
Since $v^h(\CT)> 0$, we conclude that 
\begin{equation} \label{eqn:thetaupperbound}
    \theta<\varphi.
\end{equation}
The lower bound for $\theta$ is a bit more involved. 
First, we note that $\CT$ can be covered by 
\begin{equation}\label{eqn:CTcover}
\left\lceil \frac{r}{\epsilon} \right\rceil
+1\leq \left\lceil \frac{r_2}{\epsilon} \right\rceil+1
\leq \frac{r_2}{\epsilon}+2    
\end{equation}
balls of radius $2\epsilon$ as we now somewhat informally explain. Let $l[o,x]$ denote 
the line between $o$ and $x$, and let $l[(\rho',0),x]$ denote the line between 
$(\rho',0)$ and $x.$ It is easy to show that the distance from any point $y\in l[o,x]$ 
to $l[(\rho',0),x]$ is smaller than $\rho'\leq \epsilon$ by considering the line segment
intersecting $y$ with angle $\theta$ to $l[o,x]$ (see Figure \ref{fig:triangles2}). 
\begin{figure} 
\begin{center}
\includegraphics[page=2, clip=true, trim=85 60 600 220 , width=8cm]{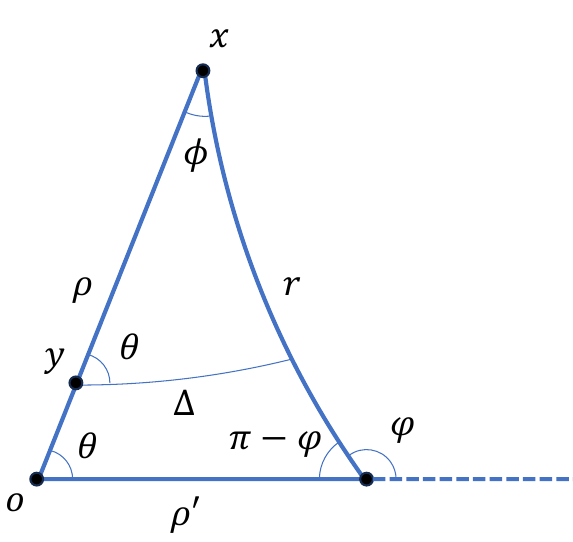} 
\caption{An illustration of why $\Delta\leq \epsilon$.} \label{fig:triangles2}   
\end{center}
\end{figure}
Letting the 
distance between $y$ and the intersection point between $l[o,x]$ and this line segment 
be denoted
by $\Delta$, we can use the hyperbolic law of sines \eqref{eqn:hypsinerule} to see that 
\[
\frac{\sinh(\Delta)}{\sin(\phi)}
\leq \frac{\sinh(r)}{\sin(\theta)}
=\frac{\sinh(\rho')}{\sin(\phi)}
\]
so that $\Delta\leq \rho'\leq \epsilon$. This shows that 
$d_h(y,l[(\rho',0),x])\leq \epsilon$ for any $y\in l[o,x],$ and a similar argument gives 
that $d_h(y,l[o,x])\leq \epsilon$ for any $y\in l[(\rho',0),x].$
Therefore, by placing balls of radius $2\epsilon$ along $l[(\rho',0),x]$ 
and with distance $\epsilon$ between the centers of consecutive balls, 
we obtain a covering of $\CT,$ and so \eqref{eqn:CTcover} follows.

Continuing, we can use \eqref{eqn:volhypball} that the hyberbolic area 
of a ball of radius $r$ is
$v^h((B^h(o,r)))=4 \pi \sinh^2(r/2)$,
to conclude that for $\epsilon>0$ small enough,
\[
v^h(\CT) \leq (r_2/\epsilon+2)v^h(B^h(o,2\epsilon))
=(r_2/\epsilon+2)(4 \pi \sinh^2(\epsilon)) 
\leq (r_2/\epsilon+2)16\pi \epsilon^2
\leq (r_2+1)16\pi \epsilon,  
\]
since $\sinh(x)\leq 2x$ for any $x>0$ small enough.
Inserting this into \eqref{eqn:triarea} we obtain
\begin{equation} \label{eqn:thetalowerbound_prel}
\varphi-\theta-\phi \leq (r_2+1)16\pi  \epsilon
\Rightarrow \theta\geq \varphi-\phi-(r_2+1)16\pi  \epsilon.    
\end{equation}
In order to get a working lower bound of $\theta,$ we therefore need to find an upper 
bound of $\phi.$ To this end, we note that by the hyperbolic sine law (i.e. \eqref{eqn:hypsinerule})
\[
\frac{\sinh(r)}{\sin(\theta)}=\frac{\sinh(\rho')}{\sin (\phi)}
\leq \frac{\sinh(\epsilon)}{\sin (\phi)}
\Rightarrow \sin(\phi)\leq \sinh(\epsilon) \frac{\sin(\theta)}{\sinh(r)}.
\]
We then use the elementary inequalities $\sinh(\epsilon)\leq \epsilon+\epsilon^2$
and $\arcsin(\epsilon)\leq \epsilon+\epsilon^2$
for $\epsilon>0$ small enough, to obtain
\begin{align} \label{eqn:phiuppest}
 \phi & \leq \arcsin\left(\sinh(\epsilon)\frac{\sin(\theta)}{\sinh(r)}\right)   
\leq \arcsin\left((\epsilon+\epsilon^2)
\frac{\sin(\theta)}{\sinh(r)}\right)\\
& \leq (\epsilon+\epsilon^2)\frac{\sin(\theta)}{\sinh(r)}
+\left((\epsilon+\epsilon^2)\frac{\sin(\theta)}{\sinh(r)}\right)^2
\leq \epsilon\frac{2}{\sinh(r_1)}, \nonumber
\end{align}
whenever $\epsilon>0$ is small enough and since $r\geq r_1>0$ by assumption. 
Inserting this into \eqref{eqn:thetalowerbound_prel} we get that
\begin{equation} \label{eqn:thetalowerbound}
\theta>\varphi-\phi-(r_2+1)16\pi \epsilon \geq \varphi-C_1(r_1,r_2) \epsilon  
\end{equation}
where $C_1(r_1,r_2)=\frac{2}{\sinh(r_1)}+(r_2+1)16\pi$ is a constant which only depends on 
(our fixed choice of) $r_1$ and $r_2.$
We can then conclude from \eqref{eqn:thetaupperbound} and \eqref{eqn:thetalowerbound} that 
$\varphi-C_1(r_1,r_2) \epsilon \leq \theta \leq \varphi$ for $\epsilon>0$ small enough,
and therefore we have that
\begin{equation} \label{eqn:thetaimpliedbounds}
(\rho',\varphi,r)\in[0,\epsilon) \times [\varphi_1,\varphi_2]
\times [r_1,r_2]  
\Rightarrow \theta\in[\varphi_1-C_1(r_1,r_2)\epsilon,\varphi_2].  
\end{equation}

Finally, we address $\phi.$ To this end, note that for fixed $\rho,$ 
we can again use that $\rho-\epsilon \leq r \leq \rho +\epsilon$ to 
see that by the mean value theorem,
\[
\left|\frac{\sinh(\rho)}{\sinh(r)}-1\right|
=\frac{1}{\sinh(r)} |\sinh(\rho)-\sinh(r)|
=\frac{\cosh(\xi)}{\sinh(r)}|\rho-r|
\leq\frac{\cosh(\xi)}{\sinh(r)} \epsilon
\]
for some $\xi\in [r-\epsilon,r+\epsilon].$ Since $r>r_1$ 
and $\cosh(x)$ is an increasing function for $x>0,$ we conclude that for 
$\epsilon>0$ small enough,
\begin{align} \label{eqn:sinhratio}
\left|\frac{\sinh(\rho)}{\sinh(r)}-1\right|
& \leq\frac{\cosh(r+\epsilon)}{\sinh(r)} \epsilon
=\frac{\cosh(r) \cosh(\epsilon)+\sinh(r)\sinh(\epsilon)}{\sinh(r)} \epsilon \\
& \leq 2\coth(r)\epsilon+2\epsilon^2 
\leq 2\coth(r_1)\epsilon+2\epsilon^2 
\leq 3\coth(r_1)\epsilon
\nonumber
\end{align}
by using the easily verifiable fact that $\coth(r)=\cosh(r)/\sinh(r)$ is decreasing in $r>0$
and that $r>r_1$ by assumption. We also used that $\cosh(\epsilon)\leq 2$ and that 
$\sinh(\epsilon)\leq 2 \epsilon$ for $\epsilon>0$ small enough.
Next, as in \eqref{eqn:phiuppest} we have that 
\begin{align*} 
 \phi& \leq (\epsilon+\epsilon^2)\frac{\sin(\theta)}{\sinh(r)}
+\left((\epsilon+\epsilon^2)\frac{\sin(\theta)}{\sinh(r)}\right)^2
\leq \epsilon\frac{\sin(\theta)}{\sinh(r)} +\epsilon^{2}\frac{\sin(\theta)}{\sinh(r)} 
+4\epsilon^2\left(\frac{\sin(\theta)}{\sinh(r)} \right)^2\\
& \leq \epsilon\frac{\sin(\theta)}{\sinh(r)}
+\left(1+\frac{4}{\sinh(r_1)}\right)\frac{\epsilon^2}{\sinh(r)}
=\left(\epsilon\frac{\sin(\theta)}{\sinh(\rho)}
+\left(1+\frac{4}{\sinh(r_1)}\right)\frac{\epsilon^2}{\sinh(\rho)}\right)\frac{\sinh(\rho)}{\sinh(r)}
\nonumber \\
& \leq \left(\epsilon\frac{\sin(\theta)}{\sinh(\rho)}
+\left(1+\frac{4}{\sinh(r_1)}\right)\frac{\epsilon^2}{\sinh(\rho)}\right)
\left(1+3\coth(r_1)\epsilon\right) \\
& \leq  \epsilon\frac{\sin(\theta)}{\sinh(\rho)}
+\frac{\epsilon^2}{\sinh(\rho)}\left(1+\frac{4}{\sinh(r_1)}+3\coth(r_1)\right)
+\frac{\epsilon^3}{\sinh(\rho)}3\coth(r_1)\left(1+\frac{4}{\sinh(r_1)}\right)
\nonumber \\
& \leq \epsilon\frac{\sin(\theta)}{\sinh(\rho)}
+\frac{\epsilon^2}{\sinh(\rho)}\left(2+\frac{4}{\sinh(r_1)}+3\coth(r_1)\right)
=\epsilon\frac{\sin(\theta)}{\sinh(\rho)}
+\frac{\epsilon^2C_2(r_1)}{\sinh(\rho)}
\end{align*}
for $\epsilon>0$ small enough, and where we used \eqref{eqn:sinhratio} in the fourth inequality.
Here, 
\[
C_2(r_1)=2+\frac{4}{\sinh(r_1)}+3\coth(r_1).
\]
We conclude that for $\epsilon>0$ small enough,
\begin{equation} \label{eqn:phiimpliedbounds}
(\rho',\varphi,r)\in[0,\epsilon)\times [\varphi_1,\varphi_2] \times [r_1,r_2]  
\Rightarrow 
\phi\in\left[0,\epsilon\frac{\sin(\theta)}{\sinh(\rho)}+\frac{\epsilon^2C_2(r_1)}{\sinh(\rho)}\right].  
\end{equation}
Using \eqref{eqn:rhoimpliedbounds}, \eqref{eqn:thetaimpliedbounds} and 
\eqref{eqn:phiimpliedbounds} we can now conclude that if
\[
(\rho',\varphi,r)\in [0,\epsilon] \times [\varphi_1,\varphi_2]\times [r_1,r_2]  
\]
then for $\epsilon>0$ small enough,
\[
(\rho,\theta,\phi) =T^{-1}((\rho',\varphi,r)) 
\in [r_1-\epsilon,r_2+\epsilon]\times [\varphi_1-C_1(r_1,r_2)\epsilon,\varphi_2] 
\times \left[0,\epsilon\frac{\sin(\theta)}{\sinh(\rho)} 
+\frac{C_2(r_1)\epsilon^{2}}{\sinh(\rho)} \right]. \nonumber
\]
Inserting this into \eqref{eqn:muprimeA} we obtain
\begin{align*}
& \mu'(\CA) 
=\frac{(b-a)}{\epsilon}
\int_{\rho=0}^\infty  \int_{\theta=0}^{2 \pi} \int_{\phi=0}^{\pi} 
I((\rho,\theta,\phi)\in T^{-1}([0,\epsilon) \times [\varphi_1,\varphi_2]\times [r_1,r_2] )) 
\frac{1}{\pi}\dd \phi \dd \theta \sinh(\rho) \dd \rho \nonumber \\
& \leq \frac{(b-a)}{\epsilon}
\int_{\rho=r_1-\epsilon}^{r_2+\epsilon}  \int_{\theta=\varphi_1-C_1(r_1,r_2)\epsilon}^{\varphi_2}  
\int_{\phi=0}^{\epsilon\frac{\sin(\theta)}{\sinh(\rho)} +\frac{C_2(r_1)\epsilon^2}{\sinh(\rho)} } 
\frac{1}{\pi}\dd \phi \dd \theta \sinh(\rho) \dd \rho \nonumber \\
& =\frac{(b-a)}{\epsilon}\frac{1+O(\epsilon)}{\pi}
\int_{\rho=r_1}^{r_2}  \int_{\theta=\varphi_1}^{\varphi_2}  
\left(\epsilon\frac{\sin(\theta)}{\sinh(\rho)} +\frac{C_2(r_1)\epsilon^2}{\sinh(\rho)}  \right)
\dd \theta \sinh(\rho) \dd \rho \nonumber \\
& =\frac{(b-a)}{\pi}\int_{\rho=r_1}^{r_2}  \int_{\theta=\varphi_1}^{\varphi_2}  \sin(\theta) \dd \theta \dd \rho 
+O(\epsilon), \nonumber
\end{align*}
so that by taking the limit as $\epsilon\to 0,$ 
we get that 
\begin{equation} \label{eqn:muprimeAupper}
\mu'(\CA)\leq \frac{(b-a)}{\pi}\int_{\rho=r_1}^{r_2}  
\int_{\theta=\varphi_1}^{\varphi_2} 
\sin(\theta) \dd \theta \dd \rho 
=\frac{1}{\pi}\int_{\rho'=a}^b \int_{r=r_1}^{r_2} 
\int_{\varphi=\varphi_1}^{\varphi_2} 
\sin \varphi \dd \varphi \dd r \dd \rho'.    
\end{equation}

In order to show \eqref{eqn:muprimeA_org}, we need to find the matching lower bound. 
Unfortunately, we cannot rely on much of what we have already done, since the bounds 
established above often rely on conditions (such as $\rho'\leq \epsilon$) which we can no 
longer assume. However, the ideas are very similar. Consider
\begin{equation} \label{eqn:rhothetaphiassumption}
 (\rho,\theta,\phi)
\in [r_1+\epsilon,r_2-\epsilon] \times [\varphi_1,\varphi_2-C_3\epsilon] 
\times \left[0,\epsilon\frac{\sin(\theta)}{\sinh(\rho)} -C_4 \epsilon^2\right],
\end{equation}
where
\[
C_3=C_3(r_1,\varphi_1,\varphi_2)
=\frac{1+\cosh(r_1)}{\sinh(r_1)}\frac{4}{\min(\sin(\varphi_1),\sin(\varphi_2))},
\]
and where 
\[
C_4=C_4(r_1,\varphi_1,\varphi_2)
=\frac{2C_3}{\min(\sin(\varphi_1),\sin(\varphi_2))\sinh(r_1)}.
\]
We claim that for $\epsilon>0$ small enough, this implies that 
\begin{equation} \label{eqn:rhoprimerphi}
(\rho',\varphi,r)
\in [0,\epsilon)\times [\varphi_1,\varphi_2] \times [r_1,r_2] .
\end{equation}
Inserting this into \eqref{eqn:muprimeA} will give us that for $\epsilon>0$ small enough,
\begin{align*}
& \mu'(\CA)=\frac{(b-a)}{\epsilon}
\int_{\rho=0}^\infty  \int_{\theta=0}^{2 \pi} \int_{\phi=0}^{\pi} 
I((\rho,\theta,\phi)\in T^{-1}([0,\epsilon) \times [\varphi_1,\varphi_2] \times [r_1,r_2])) 
\frac{1}{\pi}\dd \phi \dd \theta \sinh(\rho) \dd \rho \nonumber \\
& \geq \frac{(b-a)}{\epsilon}
\int_{\rho=r_1+\epsilon}^{r_2-\epsilon}
\int_{\theta=\varphi_1}^{\varphi_2-C_3\epsilon}  
\int_{\phi=0}^{\epsilon\frac{\sin(\theta)}{\sinh(\rho)}-C_4 \epsilon^2}
\frac{1}{\pi}\dd \phi \dd \theta \sinh(\rho) \dd \rho \nonumber \\
& =\frac{(b-a)}{\epsilon}\frac{1+O(\epsilon)}{\pi}
\int_{\rho=r_1}^{r_2}  \int_{\theta=\varphi_1}^{\varphi_2}  
\left(\epsilon\frac{\sin(\theta)}{\sinh(\rho)} -C_4 \epsilon^2\right)
\dd \theta \sinh(\rho) \dd \rho \nonumber \\
& =\frac{(b-a)}{\pi}\int_{\rho=r_1}^{r_2}  \int_{\theta=\varphi_1}^{\varphi_2}  \sin(\theta) \dd \theta \dd \rho 
-O(\epsilon), \nonumber
\end{align*}
so that by taking the limit as $\epsilon\to 0,$ 
we get that 
\[
\mu'(\CA)\geq \frac{(b-a)}{\pi}\int_{\rho=r_1}^{r_2}  
\int_{\theta=\varphi_1}^{\varphi_2} 
\sin(\theta) \dd \theta \dd \rho 
=\frac{1}{\pi}\int_{\rho'=a}^b \int_{r=r_1}^{r_2} 
\int_{\varphi=\varphi_1}^{\varphi_2} 
\sin \varphi \dd \varphi \dd r \dd \rho'.  
\]
This, together with \eqref{eqn:muprimeAupper}, proves \eqref{eqn:muprimeA} and thereby the statement. 
What is left is therefore to show that
\eqref{eqn:rhothetaphiassumption} implies \eqref{eqn:rhoprimerphi}.

To that end, we assume \eqref{eqn:rhothetaphiassumption} and start by considering $\varphi.$ 
As before, $\theta\leq \varphi$ and so \eqref{eqn:rhothetaphiassumption} implies that 
$\varphi \geq \varphi_1.$ Furthermore, 
if we apply the hyperbolic law of cosines \eqref{eqn:hypcosinerule} to $\CT,$ we see that 
\[
\cos(\varphi) =-\cos(\pi-\varphi) 
=\cos(\theta)\cos(\phi)-\sin(\theta)\sin(\phi)\cosh(\rho).   
\]
Therefore, using that $\phi\leq \epsilon \frac{\sin(\theta)}{\sinh(\rho)}
\leq \epsilon \frac{1}{\sinh(r_1)}$ by our assumption 
\eqref{eqn:rhothetaphiassumption} and the fact that $\cos(x)\geq 1-x^2/2$ for $x>0$
small enough, we see that for $\epsilon>0$ small enough,
\begin{align} \label{eqn:cosvarphithetaineq1}
|\cos(\varphi)-\cos(\theta)|
& \leq |\cos(\theta)|\cdot |\cos(\phi)-1| +|\sin(\theta)\sin(\phi)\cosh(\rho)| \\
& \leq \frac{\phi^2}{2}+\phi\cosh(\rho)
\leq \frac{\epsilon^2}{2}\left(\frac{\sin(\theta)}{\sinh(\rho)}\right)^2
+\epsilon\sin(\theta)\coth(\rho)\\
& < \epsilon \frac{1}{\sinh(r_1)}+\epsilon\coth(r_1) 
= \frac{1+\cosh(r_1)}{\sinh(r_1)}\epsilon, \nonumber
\end{align}
where we use the fact that $\coth(\rho)$ is decreasing in $\rho$ and that 
$\rho>r_1$ by assumption.
Next, observe that since $\theta\in[\varphi_1,\varphi_2-C_3 \epsilon]$ by assumption, 
it follows that
\begin{equation} \label{eqn:minsin}
\sin(\theta) \geq \min(\sin(\varphi_1),\sin(\varphi_2-C_3 \epsilon))
\geq \min(\sin(\varphi_1),\sin(\varphi_2))    
\end{equation}
for every $\epsilon>0$ small enough. Indeed, this is easily verified by considering 
the cases $\varphi_2>\pi/2$ and $\varphi_1<\varphi_2\leq \pi/2$ separately.
Furthermore, we can use that $\cos(x)$ is decreasing for $x\in[0,\pi)$
and the easily verified fact that $\varphi\leq \pi$ whenever $\theta\leq \pi$ (see Figure 
\ref{fig:triangles}) to see that
if $\varphi>\theta+C_3\epsilon,$ then for $\epsilon>0$ small enough,
\begin{align} \label{eqn:cosvarphithetaineq1contradict}
&|\cos(\varphi)-\cos(\theta)| \\
&=\cos(\theta)-\cos(\varphi)
>\cos(\theta)-\cos(\theta+C_3\epsilon) 
=\cos(\theta)-\cos(\theta)\cos(C_3\epsilon)+\sin(\theta)\sin(C_3\epsilon) \nonumber \\
& \geq -(1-\cos(C_3\epsilon))+\min(\sin(\varphi_1),\sin(\varphi_2))\frac{C_3\epsilon}{2} 
\geq  -\frac{(C_3 \epsilon)^2}{2}+\min(\sin(\varphi_1),\sin(\varphi_2))\frac{C_3\epsilon}{2} \nonumber \\
& \geq \min(\sin(\varphi_1),\sin(\varphi_2))\frac{C_3 \epsilon}{4}
=\frac{1+\cosh(r_1)}{\sinh(r_1)}\epsilon, \nonumber
\end{align}
by the definition of $C_3$ and by taking $\epsilon>0$ small enough. 
Here, we used \eqref{eqn:minsin} in the second inequality
and the fact that $\sin(x)\geq x/2$ for $x>0$ small enough. Furthermore, 
we used that $\cos(x)\geq 1-x^2/2$ for $x>0$ small enough in the third inequality. 
The inequality \eqref{eqn:cosvarphithetaineq1contradict} contradicts 
\eqref{eqn:cosvarphithetaineq1}, and so we conclude that 
\begin{equation} \label{eqn:varphiupperbound}
\varphi\leq \theta+C_3\epsilon
\leq \varphi_2,
\end{equation}
by our assumption on $\theta$ from \eqref{eqn:rhothetaphiassumption},
which verifies that $\varphi\in[\varphi_1,\varphi_2].$

Next, we turn to $\rho'$. 
Since $\theta \leq \varphi$ and since $\varphi \leq \theta+C_3\epsilon$ by
\eqref{eqn:varphiupperbound}, we can use the mean value theorem
to see that for some $\xi\in[\theta,\theta+C_3\epsilon]$
\[
|\sin(\theta)-\sin(\varphi)|=|\cos(\xi)| C_3 \epsilon \leq C_3 \epsilon.
\]
Therefore, we can use that $\varphi\in[\varphi_1,\varphi_2]$ whenever 
\eqref{eqn:rhothetaphiassumption} holds to see that
\[
\left|\frac{\sin(\theta)}{\sin(\varphi)}-1\right|
=\frac{1}{\sin(\varphi)}\left|\sin(\theta)-\sin(\varphi)\right|
\leq \frac{C_3 \epsilon}{\min(\sin(\varphi_1),\sin(\varphi_2))},
\]
and so we see that by our assumption on $\phi$ in \eqref{eqn:rhothetaphiassumption},
\begin{align}
& \sin(\phi) \frac{\sinh(\rho)}{\sin(\varphi)}
\leq \phi \frac{\sinh(\rho)}{\sin(\varphi)}
\leq \left(\epsilon\frac{\sin(\theta)}{\sinh(\rho)}-C_4\epsilon^2 \right)
\frac{\sinh(\rho)}{\sin(\varphi)} \\
& =\epsilon\frac{\sin(\theta)}{\sin(\varphi)}
-\epsilon^2 \frac{2C_3}{\min(\sin(\varphi_1),\sin(\varphi_2))\sinh(r_1)}
\frac{\sinh(\rho)}{\sin(\varphi)} \nonumber\\
& \leq \epsilon\left(1+\frac{C_3 \epsilon}{\min(\sin(\varphi_1),\sin(\varphi_2))}\right)
-\epsilon^2 \frac{2C_3}{\min(\sin(\varphi_1),\sin(\varphi_2))}
\leq \epsilon.
\nonumber
\end{align}
Using this and the fact that $\sinh^{-1}$ is an increasing function 
such that $\sinh^{-1}(\epsilon)\leq \epsilon$
for any $\epsilon>0$ small enough, we arrive at 
\begin{align*}
\rho'& =\sinh^{-1}\left(\sin(\phi) \frac{\sinh(\rho)}{\sin(\varphi)}\right)    
\leq \sinh^{-1}\left(\epsilon\right)  \leq  \epsilon.
\end{align*}
Having established that $\rho'\leq \epsilon,$ we see (as before \eqref{eqn:rhoimpliedbounds})
that $\rho-\epsilon \leq r \leq \rho+\epsilon.$ Thus, if $(\rho,\theta,\phi)$ are as in 
\eqref{eqn:rhothetaphiassumption} we conclude that 
\[
r\in[r_1,r_2].
\]
This completes the proof that \eqref{eqn:rhothetaphiassumption} implies 
\eqref{eqn:rhoprimerphi}, and thereby the lemma.
\end{proof}

For future reference, we let
\begin{equation} \label{eqn:defCL_alt}
\CL^L([\rho'_1,\rho'_2]\times[\varphi_1,\varphi_2]\times[r_1,r_2]) 
:=\{l_L(x,\phi)\in \CL^L(g_0^+): 
\rho'\in [\rho'_1,\rho'_2], \varphi \in [\varphi_1,\varphi_2], r\in[r_1,r_2]\},   
\end{equation}
where $\rho_1'<\rho_2', \varphi_1<\varphi_2$ and $r_1<r_2.$ 
Recall that $g_0^+$ is the horizontal ray starting from $o$ and
pointing to the ``right''. We observe that
\[
\CL^L([\rho'_1,\rho'_2]\times[\varphi_1,\varphi_2]\times[r_1,r_2])
\subset \CL^L(g_0^+) \subset \CL^L.
\]
This notation complements \eqref{eqn:defCLA} and will be convenient throughout the 
rest of the paper.

\medskip
We end this section with a discussion (mentioned in the Introduction) about
why $\lambda_c(L)$ should scale like $L^{-2}$ while 
$\lambda_u(L)$ should scale like $L^{-1}.$ First, it follows from 
Lemma \ref{lemma:altPoisson}
that the expected number of sticks which hit a fixed stick $l_0$ of length $L$
is of order $\lambda L \cdot L=\lambda L^2.$  Therefore, if we take 
$\lambda \geq C L^{-2}$ and choose $C<\infty$ very large, then there should be 
many sticks which intersect $l_0.$ Thus, one should be able to grow an infinite 
structure by first considering the sticks which intersected $l_0,$ and then 
consider 
the sticks which intersected these sticks and so on. Indeed, this idea forms the 
basis of the proof of Proposition \ref{prop:half-plane} which in turn gives 
us the upper bound of Theorem \ref{thm:lambdacmain}.
For this reason, one would expect that 
that the scaling for existence of an unbounded connected component would be the same 
(up to constant factors) 
as for the Euclidean case. Indeed, this is the case as Theorem \ref{thm:lambdacmain} 
shows when compared to the main result of \cite{B_22}.

However, the situation when studying uniqueness of unbounded components 
differ in a qualitative way, as we will now attempt to explain. Unfortunately, some of the steps 
in this intuitive explanation are still unproven, but we believe that it should be convincing. 
Firstly, there is an unproven connection between the stick model and the so-called
Poisson cylinder model both in Euclidean (see \cite{BT_16}) and hyperbolic spaces
(see \cite{BT_15}), in that taking the limit as $L \to \infty$ in the stick model 
should yield the cylinder model. When taking this limit, it is clear from Lemma 
\ref{lemma:altPoisson} that one needs to thin the stick model by a factor of $L^{-1},$ 
otherwise the probability, in the limit, of a cylinder hitting a bounded set, say 
$l[0,(1,0)],$ would be 0 (if thinned 
by a factor of say $L^{-(1+\epsilon)}$) or one (if thinned by a factor of say 
$L^{-1} (\log L)$). 
In the appropriate 
Euclidean limit (i.e.\ the Poisson cylinder model), the set of cylinders will always consist  
of a single connected component (see \cite{BT_16})
and so there is always uniqueness and no phase transition.
In contrast, the cylinder limit in the hyperbolic case does not always contain a unique 
unbounded component (see \cite{BT_15}). Thus, if one
takes $\lambda=c L^{-1},$ with $c$ small enough, in the hyperbolic stick model, then as 
$L \to \infty,$ one would obtain a limit with connected unbounded components 
(because of Theorem \ref{thm:lambdacmain})
but without uniqueness. 
If instead, one takes the same limit but with  $\lambda=C L^{-1},$ and with $C$ large enough,
one would obtain a limit with a unique connected component. Thus, this intuition suggests that 
\[
\lambda_u(L)\sim L^{-1},
\]
which is fundamentally different from the Euclidean case. This is the intuitive explanation for why 
we obtain the result in Theorem \ref{thm:lambdaumain}.

\section{Proof of Theorem \ref{thm:lambdacmain}} \label{sec:lambdac}
In this section we shall prove Theorem \ref{thm:lambdacmain}. The proofs of the two directions
of Theorem \ref{thm:lambdacmain} are very different, and therefore we subdivide this section 
into two subsections, proving the two directions of Theorem \ref{thm:lambdacmain} separately. 

\subsection{The lower bound of Theorem \ref{thm:lambdacmain}.}
The following result provides the lower bound of Theorem \ref{thm:lambdacmain}. The proof 
proceeds by a coupling between the stick process and a branching process.
The proof is similar to the proof of the lower bound of Theorem 3.1 in
\cite{B_22}. However, in order to obtain the value of the constant in the lower bound,
and in order to keep the paper self-contained, we will provide 
the full argument here. 

\begin{proposition} \label{prop:lambdaclowerbound}
For every $L\geq 10$ we have that
\[
\frac{\pi}{2} L^{-2} \leq  \lambda_c(L).
\]
\end{proposition}
\begin{proof}
Let $\omega^{\lambda,L}$ be as before (see Section \ref{sec:Poistick}), and let 
$(\omega^{\lambda,L}_{k,n})_{k,n\geq 1}$ be an i.i.d. collection where 
$\omega^{\lambda,L}_{k,n}$ has the same distribution as $\omega^{\lambda,L}.$
Then, consider the stick $l_L(o,0)$ and let 
$\CC_o(\omega^{\lambda,L})$ denote the connected component of $\CC\cup l_L(o,0)$ 
containing $l_L(o,0).$
The main idea is to use $\omega^{\lambda,L}$ and the sequence 
$(\omega^{\lambda,L}_{k,n})_{k,n\geq 1}$ to 
construct a set which is larger than $\CC_o(\omega^{\lambda,L})$ and which will  
be coupled to a subcritical branching process. This allows us to conclude that 
this larger set is finite from which follows that
also $\CC_o(\omega^{\lambda,L})$ is finite.
In this context, $l_{L}(o,0)$ corresponds to generation 0 of the branching process.
For convenience, we let $l_0:=l_L(o,0),$ and note that although $l_{L,0}$ or 
something similar would be a more consistent notation, this would quickly 
become cumbersome. A similar comment applies to many places below.

Let 
\[
\psi^\lambda_1=\{l\in \omega^{\lambda,L}: l \cap l_0 \neq \emptyset\}.
\]
We see that $\psi^{\lambda}_1$ consists of the sticks in $\omega^{\lambda,L}$
that intersect the generation 0 stick $l_{L}(o,0)$. Let $l_{1,1},\ldots 
l_{|\psi^{\lambda}_1|,1}$ be an enumeration of the sticks in $\psi^\lambda_1.$
We think of these sticks as generation 1. 

In order to define what will be generation 2, some care is needed. Consider first
\[
\psi^{\lambda}_{1,2}
:=\{l\in \omega^{\lambda,L} \setminus \psi^\lambda_1 : l \cap l_{1,1} \neq \emptyset \}
\cup \{l\in \omega^{\lambda,L}_{1,1}: l \cap l_{1,1}\neq \emptyset, 
l \cap l_0\neq \emptyset\}.
\]
The first of the two sets on the right hand side is the collection of sticks
in $\omega^{\lambda,L}$ that intersect the stick $l_{1,1},$
but that we did not encounter when we defined $\psi^\lambda_1.$ This guarantees that 
we are using $\omega^{\lambda,L}$ on disjoint sets when defining $\psi^\lambda_1$
and $\psi^\lambda_{1,2}.$ However, it is then also clear that in order for the 
number of lines in $\psi^\lambda_1$ and $\psi^\lambda_{1,2}$ to have the same 
distribution, we need to compensate. This is why the second set is there, since this 
is the ``missing part'' (i.e.\ it counts lines that intersect both $l_0$ and 
$l_{1,1}$). Therefore, $|\psi^\lambda_1|$ and $|\psi^{\lambda}_{1,1}|$ have the same 
distribution, and since we used $\omega^{\lambda,L}_{1,1}$ in the second set,
we are making sure that $|\psi^\lambda_1|$ and $|\psi^{\lambda}_{1,1}|$
are also independent. Lastly, we observe that any $l\in \omega^{\lambda,L}$
that intersects $l_0$ and/or $l_{1,1}$ must belong to 
$\psi^\lambda_1 \cup \psi^\lambda_{1,2}.$ 

The above paragraph explained how to handle the offspring of the first child in 
generation 1. In general, for child $k=2,\ldots,|\psi^\lambda_1|$ we define
\begin{equation}\label{eqn:defpsik2}
    \psi^\lambda_{k,2} =\{l\in \omega^{\lambda,L}\setminus 
\left(\psi^\lambda_1 \cup_{j=1}^{k-1} \psi^\lambda_{j,2}\right): 
l\cap l_{k,1} \neq \emptyset\} 
\cup 
\{l\in \omega^{\lambda,L}_{k,1}: l \cap l_{k,1} \neq \emptyset, 
l \cap \left(l_0\cup_{j=1}^{k-1} l_{j,1}\right)\neq \emptyset\}.
\end{equation}  
The idea behind this definition is the same as before. That is, in the first set 
we use $\omega^{\lambda,L}$ for those sticks we have not encountered yet, and 
then we ``pad'' the first set with the second set in order to make sure that 
$|\psi^\lambda_{k,2}|$ has the same 
distribution as $|\psi^\lambda_1|.$ We conclude that given $|\psi^\lambda_1|,$ the sequence
$|\psi^\lambda_{1,2}|,\ldots, |\psi^\lambda_{|\psi^\lambda_1|,2}|$ is an i.i.d.\ 
sequence. We then let
\begin{equation} \label{eqn:defpsi2}
   \psi^\lambda_2=\bigcup_{k=1}^{|\psi^\lambda_1|} \psi^\lambda_{k,2}
\end{equation} 
be the collection of sticks in generation 2, and we enumerate them 
$l_{1,2},\ldots l_{|\psi^\lambda_2|,2}.$ Note that if $l\in \omega^{\lambda,L}$
is such that $l$ is connected to $l_0$ using at most one other stick from 
$\omega^{\lambda,L},$ then we must have that $l\in \psi^\lambda_1 \cup \psi^\lambda_{2}.$

The idea for successive generations is the same. That is, given $\psi^\lambda_n,$ 
we define the collection $(\psi^\lambda_{k,n+1})_{1\leq k\leq |\psi^\lambda_n|}$ in the 
way analogous to \eqref{eqn:defpsik2} and let 
\[
\psi^\lambda_{n+1}=\bigcup_{k=1}^{|\psi^\lambda_n|} \psi^\lambda_{k,n+1},
\]
analogous to \eqref{eqn:defpsi2}.
If we then let $\psi^\lambda=\bigcup_{n=1}^\infty \psi^\lambda_n$ and 
\[
\CC_o(\psi^\lambda)=l_0 \bigcup_{l \in \psi^\lambda} l,
\]
(which we note is a connected component by definition) it follows from our 
construction that
\begin{equation} \label{eqn:CCoinclusion} 
 \CC_o(\omega^{\lambda,L}) \subset \CC_o(\psi^\lambda).   
\end{equation}

By construction, the sequence $(|\psi^\lambda_n|)_{n\geq 1}$ corresponds to the 
generational sizes of a Galton-Watson tree. What is left is to show that 
$\BE[|\psi^\lambda_1|]\leq 1$ whenever $\lambda\leq \frac{2}{\pi} L^{-2},$ since then this tree is 
subcritical and will almost surely be finite. This in turn implies that 
$\CC_o(\psi^\lambda)$ is finite, and by \eqref{eqn:CCoinclusion} 
we can conclude that also $\CC_o(\omega^{\lambda,L})$ is finite. 

Recall from \eqref{eqn:defCLA} that $\CL^L(l_0)$ is the set of sticks which intersect 
$l_0.$ We can now use Lemma \ref{lemma:altPoisson} to see that 
\begin{equation} \label{eqn:Epsi}
\BE[|\psi^\lambda_1|] = \lambda\mu(\CL^L(l_0))
= \frac{\lambda}{\pi}\int_{\rho'=0}^L \int_{-L/2}^{L/2} \int_{0}^{\pi} 
\sin \varphi \dd \varphi \dd r \dd \rho'
=\frac{2\lambda }{\pi}L^2.    
\end{equation}
Therefore, if we pick $\lambda\leq \frac{\pi}{2}L^{-2},$ then $\BE[|\psi^\lambda_1|] \leq 1,$ 
and as explained, $\CC_o(\omega^{\lambda,L})$ is then finite almost surely.
By standard Poisson process arguments, $\CC(\omega^{\lambda,L})$ cannot contain any unbounded 
components, and we therefore conclude that 
\[
\lambda_c(L)\geq \frac{\pi}{2}L^{-2}.
\]
\end{proof}

\subsection{The upper bound of Theorem \ref{thm:lambdacmain}.}
The purpose of this subsection is to prove the upper bound of 
Theorem \ref{thm:lambdacmain} which will follow from Proposition 
\ref{prop:half-plane} below. Let 
\[
H_0:=\{(\rho,\theta)\in \BH^2: \theta \in[0,\pi/2]\cup [3\pi/2,2\pi), 0\leq \rho<\infty\},
\]
denote the ``right-hand'' half-plane of $\BH^2.$ Then, we let
\begin{equation} \label{eqn:defCH0}
\CC_{|H_0}:= \bigcup_{l_L\in \omega^{\lambda,L}: l_L\subset H_0} l_L,   
\end{equation}
and let $\left(\CC_{|H_0}\right)_{l_0}\subset \CC_{|H_0}\cup l_0$ denote the connected component of 
$\CC_{|H_0} \cup l_0$ which contains $l_0.$ We have the following result, which 
concerns
percolation in a half-plane. 
\begin{proposition} \label{prop:half-plane}
For any $0<p<1,$ there exists a constant $0<C(p)<\infty$ such that for any $L$ large enough, 
and any $\lambda \geq CL^{-2}$ we have that 
\begin{equation}\label{eqn:Plargerthanp}
 \BP\left(\left(\CC_{|H_0}(\omega^{\lambda,L})\right)_{l_0} \textrm{ is unbounded }\right)>p.   
\end{equation}
Furthermore, if $C\geq \frac{32 \pi}{\sqrt{3}-1},$ then 
\eqref{eqn:Plargerthanp} holds with $p>0.$
\end{proposition}
\noindent
Clearly, Proposition \ref{prop:half-plane} implies the upper bound of 
Theorem \ref{thm:lambdacmain}.

Before we give the proof of Proposition \ref{prop:half-plane}, we shall informally explain the idea behind 
it, see also Figures  \ref{fig:embeddtree1} and \ref{fig:embeddtree2} for 
illustrations. 
We start with the stick $l_0=l_L(o,0)$. We will then consider 
the set of sticks $l_L(x,\phi)=l_L(\rho',\varphi,r)$ (using the notation established
before Lemma \ref{lemma:altPoisson}) satisfying the following 
three properties:
\begin{enumerate}
    \item $l_L(\rho',\varphi,r)$ hits $l_0$ between $(L/4,0)$ and $(L/2,0),$ 
    i.e.\ $\rho'\in[L/4,L/2].$

    \item $l_L(\rho',\varphi,r)$ hits $l_0$ at an angle in the interval 
    $[\pi/6,\pi/3],$ i.e.\  $\varphi\in[\pi/6,\pi/3].$

    \item $l_L(\rho',\varphi,r)=l_L(x,\phi)$ is such that $d_h(x,l_L(x,\phi)\cap l_0)\geq L/4$
    where $x=x(\rho',\varphi,r).$ That is, $r\in[L/4,L/2].$ 
\end{enumerate} 
Recall the notation  $\CL^L([L/4,L/2]\times[\pi/6,\pi/3]\times[L/4,L/2])$ from
\eqref{eqn:defCL_alt} and observe that by Lemma \ref{lemma:altPoisson} we have that 
\begin{equation}\label{eqn:muofsticks}
\mu(\CL^L([L/4,L/2]\times[\pi/6,\pi/3]\times[L/4,L/2])) 
=\frac{1}{\pi}\int_{L/4}^{L/2} \int_{\pi/6}^{\pi/3} \int_{L/4}^{L/2} 
\sin \varphi  \dd r  \dd \varphi \dd \rho'
=\frac{\sqrt{3}-1}{32\pi }L^2. 
\end{equation} 
Therefore, if $\lambda\geq C L^{-2}$ with $C<\infty$ large, we will be able 
to conclude that with probability close to one, there exists a line 
segment $l_1\in \omega^{\lambda,L}$ satisfying 
these three conditions. 
Similarly, there is with probability close to one another stick
$l_{-1}\in \omega^{\lambda,L}$ satisfying the corresponding conditions 
reflected in 
the horizontal axis (see Figure \ref{fig:embeddtree2}). That is (using 
the natural notation),
\[
l_{-1}\in \CL^L([L/4,L/2]\times[\pi-\pi/3,\pi-\pi/6]\times[-L/4,-L/2]).
\]
Given such a line segment $l_1$ we then consider the half-plane (see again Figures
\ref{fig:embeddtree1} or \ref{fig:embeddtree2}) $H_1$ defined by the geodesic 
which passes through the center of $l_1$ and which is orthogonal 
to $l_1,$ and 
we prove in Lemma \ref{lemma:half-planes}
that $H_1 \subset H_0.$ Similarly, having defined $H_{-1}$ in the obvious way, we 
prove in Lemma \ref{lemma:half-planes}
that $H_{-1}\subset H_0$ and furthermore that $H_1 \cap H_{-1}=\emptyset.$
We can now repeat the procedure just described on the part of the line 
segment $l_1$ ($l_{-1}$) which belongs to $H_1$ ($H_{-1}$), attempting to 
find the line segments $l_{1,1}$ and $l_{1,-1}$ ($l_{-1,1}$ and $l_{-1,-1}$).
Continuing, this construction can be coupled to a Galton-Watson tree which will be 
supercritical whenever $\lambda\geq C L^{-2}$ with $C<\infty$ large enough.

\begin{figure}[h] 
\begin{center}
\includegraphics[page=1, clip=true, trim=90 80 340 60 , width=12cm]{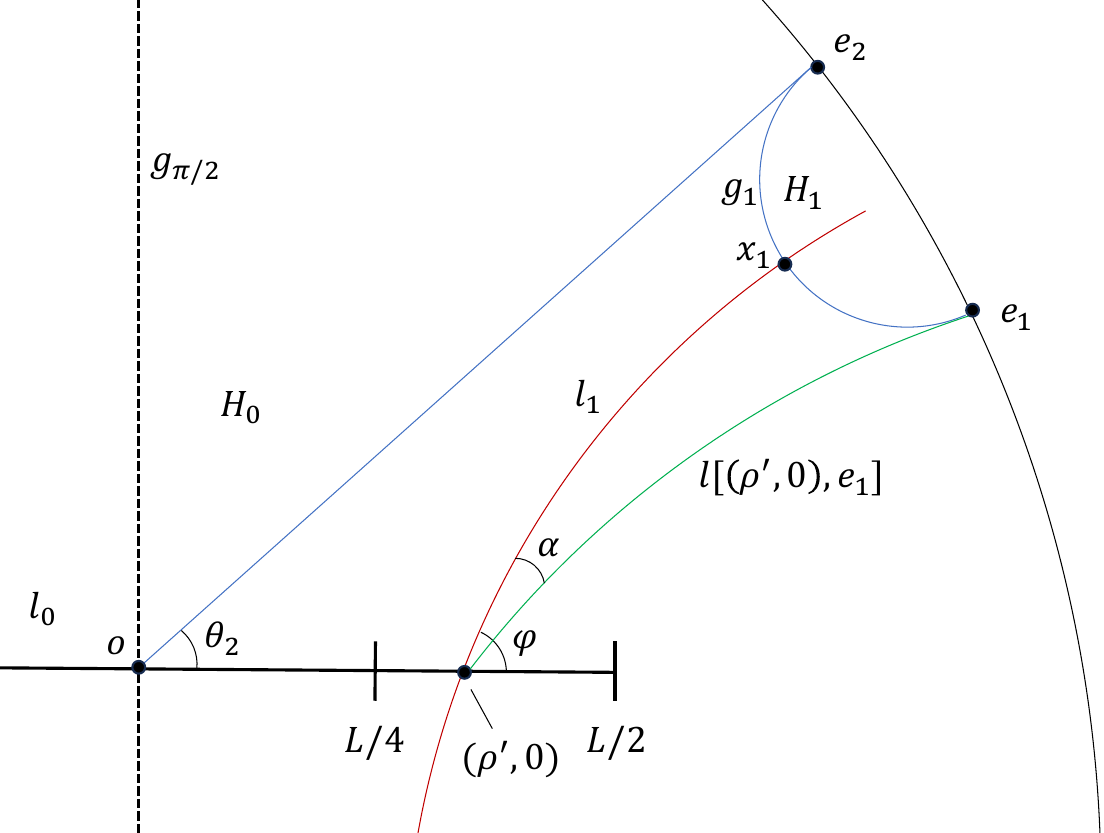}
\caption{The horizontal black line is $l_0.$ In red is $l_1,$ while 
$g_1$ and $l[o,e_2]$ are marked in blue, and $l[(\rho',0),e_1]$ in green. We can also see
the half-planes $H_0$ and $H_1$ and the angles $\varphi,\theta_2,\alpha$ 
mentioned in the proof of Lemma \ref{lemma:half-planes}.
The large black arc is part of $\partial \BH^2.$ The line $l[o,e_1]$
is not in this picture.} \label{fig:embeddtree1}
\end{center}
\end{figure}

\begin{figure}[h] 
\begin{center}
\includegraphics[page=2, clip=true, trim=90 20 340 60 , width=12cm]{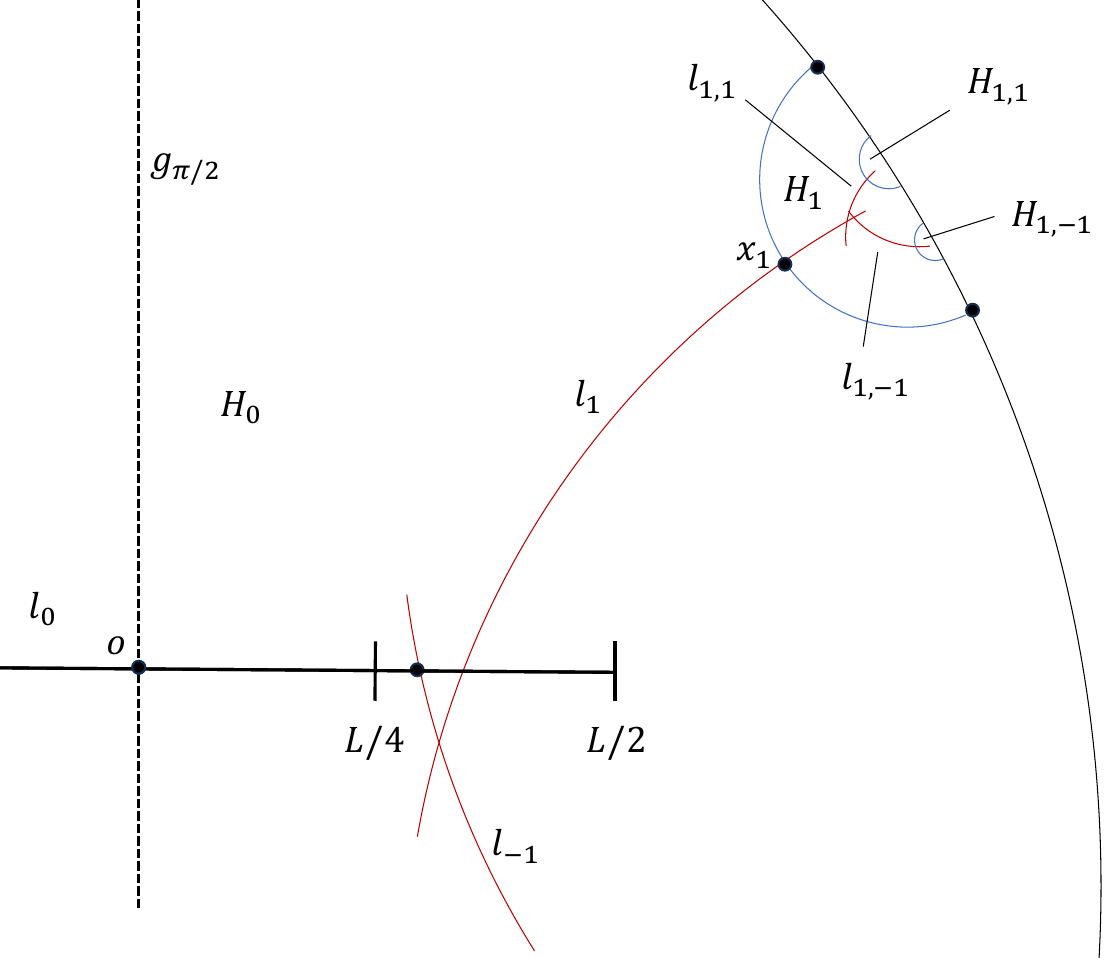} 
\caption{In this picture we can see $l_1,l_{-1},l_{1,1}$ and $l_{1,-1}$ all marked in red. 
In addition, we see the corresponding half-planes $H_1,H_{1,1}$ and $H_{1,-1}.$ The 
half-plane $H_{-1}$ is outside of the picture. 
The large black arc is part of $\partial \BH^2.$} \label{fig:embeddtree2}
\end{center}
\end{figure}

Before we can prove Proposition \ref{prop:half-plane}, we need to
address Lemma \ref{lemma:half-planes}. Let $l_0=l_L(o,0)$ and let 
$l[(3L/4,0),(L,0)] \subset l_0$ be the line segment of length $L/4$ centred at 
$(7L/8,0)$ and with angle $0$ (recall the definition of $l[x,y]$ from 
Section \ref{sec:modelsanddef}). Then, let 
\[
l_{1}\in \CL^L([L/4,L/2]\times[\pi/6,\pi/3]\times[L/4,L/2])
\]
and
\[
l_{-1}\in \CL^L([L/4,L/2]\times[\pi-\pi/3,\pi-\pi/6]\times[-L/4,-L/2])
\]
be arbitrary. Next, we consider 
$l_1=l_L(x_1,\theta_1)$ 
as above, and we let 
$g_1$ be the geodesic through $x_1$ which is perpendicular to $l_1$
(see again Figure \ref{fig:embeddtree1}).
Clearly, 
$g_1$ divides $\BH^2$ into two half-planes, and we let $H_1$ be the one 
which do not contain 
the origin $o.$ Similarly, we define $H_{-1}$ using $l_{-1}=l_L(x_{-1},\theta_{-1})$  and $g_{-1}.$

\begin{lemma} \label{lemma:half-planes}
For any $l_1$ and $l_{-1}$ as above, and any $0<L<\infty$ large enough, we have that 
\[
H_1 \subset H_0, \ \  H_{-1} \subset H_0 \textrm{ and } H_1 \cap H_{-1} =\emptyset.
\]
Furthermore, 
\[
l_1 \subset H_0 \textrm{ and } l_{-1} \subset H_{0}.
\]
\end{lemma}
\begin{proof}
We start by noting that the first statement follows if we can prove that 
\begin{equation} \label{eqn:g1include}
   g_1 \subset H_0^+:=\{(\rho,\theta)\in \BH^2: 0< \theta < \pi/2, 0\leq \rho<\infty\}, 
\end{equation}
since it then clearly follows that 
$H_1 \subset H_0^+ \subset H_0$ and (using obvious notation) 
$H_{-1} \subset H_0^{-} \subset H_0$ and furthermore, 
$H_0^+ \cap H_0^-=\emptyset$ by definition.

In order to show \eqref{eqn:g1include}, consider the two endpoints 
$e_1=e_1(g_1),e_2=e_2(g_1) \in \partial \BH^2$ of $g_1$. 
Let $l[o,e_1]$ and $l[o,e_2]$  be the infinite straight lines between 
$o,$ $e_1$ and $o,$ $e_2$ respectively.
Then, let $\theta_1=\theta_1(e_1)$ denote the angle between $l_0$
and $l[o,e_1]$, and define $\theta_2=\theta_1(e_2)$ correspondingly by using 
$l[o,e_2]$.

Consider next the triangle defined by the three points 
$(\rho',0),x_1$ and $e_1$ 
(see again Figure \ref{fig:embeddtree1}). 
Let $\alpha$ denote the angle opposite the edge between $x_1$ and $e_1$
(see again Figure \ref{fig:embeddtree1}),
and if we note that the angle $\beta$ between 
$g_1$ and $l_1$ equals $\pi/2$ by definition, then by the hyperbolic law of cosines 
\eqref{eqn:hypcosinerule_infinite}
\[
\cosh(d_h((\rho',0),x_1))=\frac{1+\cos(\alpha)\cos(\beta)}{\sin(\alpha)\sin(\beta)}
=\frac{1}{\sin(\alpha)}.
\]
By definition, $d_h((\rho',0),x_1)\geq L/4$ and therefore,
\[
\sin(\alpha)=\frac{1}{\cosh(d_h((\rho',0),x_1))}\leq \frac{1}{\cosh(L/4)}
\leq 2e^{-L/4},
\]
from which it follows that 
\[
\alpha  \leq \sin^{-1}\left(\frac{1}{\cosh(L/4)}\right)\leq 4e^{-L/4},
\]
where the last inequality follows whenever $L$ is large enough. 
Since $\varphi(l_1) \geq \pi/6$ by assumption, it follows that the angle between 
$l[(\rho',0),e_1]$ and $l_0$ is at least $\pi/6-\alpha\geq \pi/6-4e^{-L/4}>0$ 
for $L$ large enough. It then easily follows that also $\theta_1>0.$
In a similar way, the angle between $l_0$ and $l[(\rho',0),e_2]$ is at most
$\pi/3+4e^{-L/4}<\pi/2$ for $L$ large enough, from which it follows that 
$\theta_2< \pi/2.$ We conclude that $g_1 \subset H_0^+$ as desired.

For the second statement, let $l_1=l_1^+ \cup l_1^-$ where $l_1^+:=l_1 \cap H_1$
(see also Figure \ref{fig:triangles3}). 
By definition, $l_1^+\subset H_1 \subset H_0$ where the second inclusion follows from the first statement
of this lemma. Next, observe that $l_1^-$ is such that $(\rho',0)\in l_1^-$ and that 
the hyperbolic length of $l_1^-$ equals $L/2.$
Consider now the ``vertical axis'' $g_{\pi/2}=\partial H_0 \cap \BH^2.$ 
Assume for contradiction that $l_1 \cap H_0^c \neq \emptyset$ so that 
$l_1^-\cap g_{\pi/2}=y$ for some $y\in g_{\pi/2}.$ 
Applying the hyperbolic law of sines to the triangle with corner points $o,(\rho',0)$ and $y$ shows that 
\[
\frac{\sinh(d_h(y,(\rho',0)))}{\sin \pi/2}
=\frac{\sinh(d_h(o,(\rho',0)))}{\sin \gamma}> \sinh(d_h(o,(\rho',0))),
\]
where $\gamma$ is the angle between $l[o,y]$ and $l[(\rho',0),y]$ (see again Figure \ref{fig:triangles3})
 and since trivially, $y\neq o$.
\begin{figure}[h] 
\begin{center}
\includegraphics[page=3, clip=true, trim=60 180 600 50 , width=8cm]{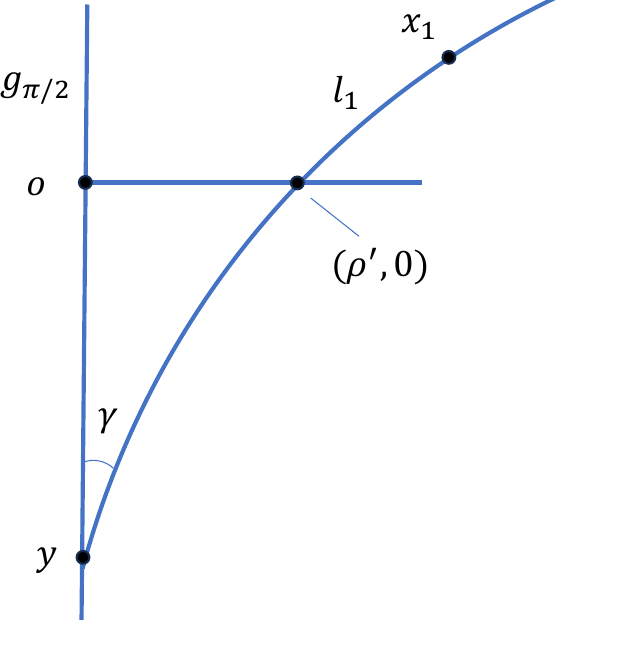} 
\caption{An illustration of the impossibility that $y\in g_{\pi/2} $.} \label{fig:triangles3}   
\end{center}
\end{figure}
Therefore, $d_h(y,(\rho',0)) > d_h(o,(\rho',0))\geq L/4$ by assumption. 
It follows that $
|l^-_1|\geq d_h(x_1,(\rho',0))+d_h((\rho',0),y)>L/4+L/4=L/2$ contradicting that the length of 
$l^-_1$ equals $L/2.$
\end{proof}

We can now turn to the proof of the main result of this subsection.
\begin{proof}[Proof of Proposition \ref{prop:half-plane}]
We will couple the stick process with a Galton-Watson process in such a way that 
the survival of the Galton-Watson process implies the existence of an unbounded connected component
in the stick process.
We will see that if we consider the stick process with parameters 
$\lambda$ and $L,$
it suffices to let $\lambda \geq \frac{\sqrt{3}-1}{32 \pi}L^{-2}$ 
in order to guarantee that the mentioned
Galton-Watson process is supercritical.

We shall start by describing the coupling and then address when the coupled Galton-Watson 
process is supercritical.
Let $l_0=l_L(o,0)$ be as before, we think of this as the root in the corresponding 
Galton-Watson tree. 
If 
\[
\omega^{\lambda,L}\cap \CL^L([L/4,L/2]\times [\pi/6,\pi/3]\times [L/4,L/2]) \neq \emptyset,
\]
then pick a stick $l_1$ in this set using some arbitrary rule (if there is more than 
one to choose between). 
Similarly, if 
\[
\omega^{\lambda,L}\cap \CL^L([L/4,L/2]\times [\pi-\pi/3,\pi-\pi/6]\times [-L/4,-L/2]) \neq \emptyset,
\]
pick some stick $l_{-1}$ again using some rule. It is clear that the existence of 
$l_1$ and $l_{-1}$ are independent since the center-points of these sticks belong to disjoint 
parts of $\BH^2.$ We think of $l_1$ and $l_{-1}$ (if they exist) as generation 1 in the 
Galton-Watson tree. 

We will let 
\begin{align} \label{eqn:qdef}
 q& =\BP(\omega^{\lambda,L}\cap \CL^L([L/4,L/2]\times [\pi/6,\pi/3]\times [L/4,L/2])\neq \emptyset) \\
& =1-\exp(-\lambda \mu(\CL^L([L/4,L/2]\times [\pi/6,\pi/3]\times [L/4,L/2]))) =1-\exp\left(-\lambda \frac{\sqrt{3}-1}{32 \pi}L^2\right), \nonumber
\end{align}
where we used \eqref{eqn:muofsticks} in the last equality. Thus, $q$ is the probability 
of finding such a stick $l_1,$ and we note that by symmetry, the probability of finding 
a stick $l_{-1}$ must also be $q.$
If there are no members in generation 1, then the coupling ends, and of course the Galton-Watson 
tree is then finite. If however $l_1$ belongs to generation 1, consider the unique 
orientation-preserving isometry 
$\CI_1$ mapping $l_1$ onto $l_0$ and $H_1$ onto $H_0.$ Then, we look for sticks 
$l_{1,1},l_{1,-1} \in \omega^{\lambda,L}$ such that 
\[
\CI_1(l_{1,1})\in \CL^L([L/4,L/2]\times [\pi/6,\pi/3]\times [L/4,L/2])
\]
and 
\[
\CI_1(l_{1,-1})\in \CL^L([L/4,L/2]\times [\pi-\pi/3,\pi-\pi/6]\times [-L/4,-L/2]).
\]
Observe that by Lemma \ref{lemma:half-planes},
$\CI_1(l_{1,1})\subset H_0$ so that $l_{1,1}\subset H_1.$ Therefore, 
$l_{1,1}$ does not belong to $\CL^L([L/4,L/2]\times [\pi/6,\pi/3]\times [L/4,L/2])$ 
(but of course $\CI_1(l_{1,1})$ does) and so $l_{1,1}$
would not have been encountered when exploring $\omega^{\lambda,L}$ looking for $l_1.$ Therefore, 
given $l_1,$ the probability of finding $l_{1,1}$ is again $q.$ Furthermore, $l_{1,1}$ and $l_{1,-1}$
are independent for the same reason that $l_1$ and $l_{-1}$ was independent.
Similarly, if $l_{-1}$ belongs to generation 1, 
we consider the unique orientation-preserving isometry $\CI_{-1}$ which maps $l_{-1}$ onto $l_0$ and 
$H_{-1}$ onto $H_0.$ Then, we look for sticks $l_{-1,1},l_{-1,-1} \in \omega^{\lambda,L}$ 
such that 
\[
\CI_{-1}(l_{-1,1})\in \CL^L([L/4,L/2]\times [\pi/6,\pi/3]\times [L/4,L/2])
\]
and 
\[
\CI_{-1}(l_{-1,-1})\in \CL^L([L/4,L/2]\times [\pi-\pi/3,\pi-\pi/6]\times [-L/4,-L/2]).
\]
The sticks (if they exist) are subsets of 
$H_{-1},$ and since Lemma \ref{lemma:half-planes} tells us that $H_1 \cap H_{-1}=\emptyset,$
it follows that the existence of $l_{-1,1}$ and $l_{-1,-1}$ are independent of the existence
of $l_{1,1}$ and $l_{1,-1}.$ The potential sticks in generation 2 are therefore 
$l_{1,1},l_{1,-1},l_{-1,1},l_{-1-1}$, and they belong to generation 2
independently and with probability $q,$ given the members of generation 1.

Given the members of generation 2 (if any exists), we now proceed in the obvious way to look for 
$l_{1,1,1},\ldots l_{-1,-1,-1},$ the potential members of generation 3. Then, we proceed 
further to find generation $n$ given that there were members of generation $n-1.$ This 
procedure provides the coupling of the stick process with a Galton-Watson tree. 
Furthermore, it is clear from the construction that the number of offspring 
from any individual is 
binomially distributed with parameters 2 and $q.$ 
The expected number of offspring in this 
Galton-Watson tree is clearly $2q,$ and we have from \eqref{eqn:qdef} that if 
\[
\lambda \geq \frac{32 \pi}{\sqrt{3}-1}L^{-2},
\]
then
\[
2q = 2\left(1-\exp\left(-\lambda \frac{\sqrt{3}-1}{32 \pi}L^2\right)\right)
\geq 2\left(1-\exp\left(-1\right)\right)>1.
\]
Therefore, the corresponding Galton-Watson process is supercritical and survives 
with positive probability, which shows the second statement of the proposition.
Finally, it is well known that the survival probability of the 
Galton-Watson process
goes to 1 as $q \uparrow  1$, and we see that for $C<\infty$ large enough, 
$q$ from \eqref{eqn:qdef} can be made arbitrarily close to one. This shows the 
first statement of the proposition and concludes the proof.
\end{proof}
\noindent
{\bf Remark:} It is clear from the proof of Proposition \ref{prop:half-plane} that 
the constant $\frac{\sqrt{3}-1}{32 \pi}$ can be improved upon. For instance, the conditions 
on $l_1,l_{-1}$ are unnecessarily restrictive, and $q>1$ for smaller values of $\lambda$
than those stated. It seems conceivable that by additional effort or perhaps all together 
new ideas, one might be able to remove these restrictions all together. In that way, one 
might find that the upper bound in fact matches the lower bound 
$\frac{\pi}{2}L^{-1}$. Unfortunately, we could 
not at this point accomplish this, and so instead of improving the bound slightly, we opted 
for relative simplicity.

\medskip

We have the following corollary to Proposition \ref{prop:half-plane}. 
Since the result is intuitively obvious given Proposition \ref{prop:half-plane}, and since 
the proof is 
very similar to that of Proposition \ref{prop:half-plane}, we only provide a sketch.
\begin{corollary} \label{cor:seedperc}
For any $\delta>0$ and any $0<p<1$, there exists a constant $0<C(p,\delta)<\infty$ such that 
for any $L$ large enough, and any $\lambda \geq C L^{-2}$
\[
\BP(\left(\CC_{|H_0}\right)_{l[o,(\delta L,0)]} \textrm{is unbounded}) \geq p.
\]
\end{corollary}
\begin{proof}[Sketch of Proof]
The only difference between this statement and Proposition \ref{prop:half-plane} is the 
starting condition. Here, we consider $l[o,(\delta L,0)]$ instead of $l[o,(L/2,0)].$ In order 
to compensate for this, it suffices to find an auxiliary line $l_L\in \omega^{\lambda,L}$ 
such that $l_L \subset H_0$ and which intersects $l[o,(\delta L,0)]$ at a small angle
so that it essentially points ``straight ahead''. By taking $C(p,\delta)<\infty$ large enough,
it is clear that such an auxiliary line exists with probability arbitrarily close to one.
Then, this auxiliary line can be used 
as the new starting condition.
\end{proof}

The following lemma will be used in Section \ref{sec:lambdau}, but since it concerns 
percolation in a half-plane, we choose to state and prove it here.
\begin{lemma} \label{lemma:half-planeperc}
For any $\lambda>0$ and $0<L<\infty$ such that  
\[
\BP(\CC_{|H_0}(\omega^{\lambda,L}) \textrm{ contains an unbounded connected component})>0,
\]   
we have that in fact 
\[
\BP(\CC_{|H_0}(\omega^{\lambda,L}) \textrm{ contains an unbounded connected component})=1.
\]
\end{lemma} 
\begin{proof}
Consider a sequence $H_1,\ldots, H_N \subset H_0$ of disjoint half-planes. By invariance
we have that 
\begin{align*}
& \BP(\CC_{|H_0}(\omega^{\lambda,L}) \textrm{ contains an unbounded connected component}) \\
& \geq 1-
\prod_{k=1}^N \BP(\CC_{|H_k}(\omega^{\lambda,L}) \textrm{ does not contain an unbounded connected component})  \\
&=1-\BP(\CC_{|H_0}(\omega^{\lambda,L}) \textrm{ does not contain an unbounded connected component})^N,
\end{align*}
from which the statement follows by letting $N \to \infty.$
\end{proof}

We end this section with a remark mentioned in the introduction.

\noindent
{\bf Remark:}
We believe that Proposition \ref{prop:lambdaclowerbound} could be generalized to $\BH^d$
where $d\geq 3$ in a fairly straightforward manner. Of course, in $\BH^d$ one would have to consider sticks
of non-zero width (say width 1) in order for the sticks to be able to intersect. The coupling 
argument of Proposition \ref{prop:lambdaclowerbound} would still go through, but the calculation in 
\eqref{eqn:Epsi} would have to be adjusted. However, since both the ``target'' stick and the sticks 
that could hit it have length $L,$ the result would presumably be the same when $d\geq 3.$
When it comes to Proposition \ref{prop:half-plane} there is no reason why the tree-like embedding
should not work also for $\BH^d$ where $d\geq 3.$ The calculations would presumably increase 
in complexity, but the essentials such as \eqref{eqn:qdef} would remain. Thus, we conjecture 
that some version of Theorem \ref{thm:lambdacmain} holds also for $\BH^d$ when $d\geq 3.$

\section{Proof of Theorem \ref{thm:lambdaumain}} \label{sec:lambdau}
In this section we will prove Theorem \ref{thm:lambdaumain}. We establish the two 
directions of Theorem \ref{thm:lambdaumain} in two subsections. For both directions, 
the quantity
\begin{equation} \label{eqn:Vdef}
\CV:=\BH^2 \setminus \CC,    
\end{equation}
which is the vacant set, will play an important role.

\subsection{The lower bound of Theorem \ref{thm:lambdaumain}}
In this section we will establish the lower bound on $\lambda_u(L)$  
in the following proposition.
\begin{proposition} \label{prop:lambdaulower}
For any $0<L<\infty$ large enough, we have that $\lambda_u(L)
\geq \frac{\pi}{2}L^{-1}.$
\end{proposition}
The main idea of this subsection is fairly straightforward. We will establish 
that for $\lambda<\frac{\pi}{2}L^{-1},$ the vacant set $\CV$ a.s.\,contains geodesics.  
These geodesics ``separate'' half-planes that because of Lemma \ref{lemma:half-planeperc}
and Proposition \ref{prop:half-plane} must contain unbounded connected components, 
and therefore there are more than one such component. 

\medskip

The first step is to prove the following lemma.
\begin{lemma} \label{lemma:CVgeodesics}
For any $0<L<\infty$ and any $\lambda<\frac{\pi}{2}L^{-1},$ we have that $\CV$ contains 
geodesics with probability one.
\end{lemma}
The proof relies on a result from \cite{BJST_09} which we now present.
A random closed subset $\CZ\subset \BH^2$ is said to be 
{\em well behaved} if it satisfies the 
following conditions.
\begin{enumerate}
    \item $\CZ$ is invariant in distribution under isometries of $\BH^2.$

    \item The set $\CZ$ satisfies positive correlations, i.e. for every $f_1,f_2$ 
    such that $f_1,f_2$ are increasing and bounded functions of $\CZ$ it holds that
    \[
    \BE[f_1(\CZ)f_2(\CZ)] \geq \BE[f_1(\CZ)]\BE[f_2(\CZ)].
    \]

    \item There exists some $R_0<\infty$ such that $\CZ$ satisfies independence at distance $R_0.$
    That is, for every $A,A'\subset \BH^2$ such that 
    $d_h(a,a')\geq R_0 \ \forall \ (a,a')\in A \times A'$ we have that
    $\CZ\cap A$ and $\CZ\cap A'$ are independent.

    \item The expected number of connected components in $B^h(o,1)\cap \partial \CZ$ is 
    finite.

    \item The expected length of $B^h(o,1)\cap \partial \CZ$ is finite.

    \item $\BP(B^h(o,1))\subset \CZ)>0$.
\end{enumerate}
Next, let 
\begin{equation} \label{eqn:alphadef}
\alpha:=\lim_{R \to \infty} -\frac{\log \BP(l[o,(R,0)]\subset \CZ)}{R},    
\end{equation}
which exists by Lemma 3.3 of \cite{BJST_09}. Then, Lemma 3.5 of \cite{BJST_09} states that if $\alpha<1,$
then a.s. $\CZ$ contains geodesics.

We would like to apply this result to $\CV=\BH^2 \setminus \CC$, and conclude that 
it contains geodesics. However, since $\CV$ is not a closed set 
we will have to do some work. To that end, let for any $A\subset \BH^2$ 
\[
A^\epsilon:=\{x\in \BH^2: d_h(x,A)<\epsilon\},
\]
denote the $\epsilon$-enlargement of $A.$ Then, let
$\CV_\epsilon$ to be the complement of an enlarged version of $\CC$
defined by
\[
\CV_\epsilon
:=\BH^2 \setminus \left(\bigcup_{l_L\in \omega^{\lambda,L}} l_L^\epsilon\right).
\]
Defined in this way, $\CV_\epsilon$ is a closed set, and clearly $\CV_\epsilon \subset \CV$ so 
if $\CV_\epsilon$ contains geodesics, then so does $\CV$. 

\begin{proof}[Proof of Lemma \ref{lemma:CVgeodesics}]
It is easy to verify that $\CV_\epsilon$ is well behaved in the sense defined above. 
Indeed, condition 1 is satisfied since 
$\CC$ is invariant, while condition 2 follows from Lemma 2.1 of \cite{J_84}. Condition 3 
clearly holds since the sticks have length $L,$ while 4-6 are trivial. Therefore, we only 
need to show that whenever $\lambda<\frac{\pi}{2}L^{-1},$ it follows that 
$\alpha=\alpha(\lambda,\CV_\epsilon)$ of \eqref{eqn:alphadef} is strictly smaller 
than one for $\epsilon>0$ small enough.

First, let $\CV_\epsilon$ be as above and note that for any $\epsilon>0$ 
we have that 
\[
\{l_L^\epsilon(x,\phi)\cap l[o,(R,0)]\neq \emptyset\}
=\{l_L(x,\phi)\cap l[o,(R,0)]^\epsilon\neq \emptyset\},
\]
so that (recall the definition \eqref{eqn:defCLA} of $\CL^L(A)$) 
\[
\BP(l[o,(R,0)]\subset \CV_\epsilon)=\BP(\CC \cap l[o,(R,0)]^\epsilon\neq \emptyset)
=\exp(-\lambda \mu(\CL^L(l[o,(R,0)]^\epsilon))).
\]
Next, observe that
\[
\mu(\CL^L(l[o,(R,0)]^\epsilon))
\leq \mu\left(\bigcup_{k=0}^{\lfloor R \rfloor} \CL^L(l[(k,0),((k+1),0)]^\epsilon))\right)
\leq (R+1)\mu(\CL^L(l[o,(1,0)]^\epsilon)),
\]
by subadditivity and invariance of $\mu.$ 
Clearly,
\[
\lim_{\epsilon \to 0} \mu(\CL^L(l[o,(1,0)]^\epsilon))=\mu(\CL^L(l[o,(1,0)])),
\]
and so we conclude that for any $\delta>0$ there exists $\epsilon>0$ such that 
\begin{align*}
& \BP(l[o,(R,0)]\subset \CV_\epsilon) =\exp(-\lambda \mu(\CL^L(l[o,(R,0)]^\epsilon))) \\
& \geq \exp(-\lambda (R+1)\mu(\CL^L(l[o,(1,0)]^\epsilon)))  
\geq \exp(-\lambda (R+1)(1+\delta)\mu(\CL^L(l[o,(1,0)]))),      
\end{align*}
for every $R>0.$ We conclude that for any $\delta>0,$ we have that for every 
$\epsilon>0$ small enough, 
\[
\alpha(\lambda,\CV_\epsilon)
=\lim_{R \to \infty} - \frac{\BP(l[o,(R,0)]\subset \CV_\epsilon)}{R}
 \leq  \lim_{R \to \infty}  \frac{\lambda (R+1)(1+\delta)\mu(\CL^L(l[o,(1,0)]))}{R}
=  \frac{2\lambda}{\pi}(1+\delta) L,
\]
where we used Lemma \ref{lemma:altPoisson} to see that (recall the notation \eqref{eqn:defCL_alt})
\[
\mu(\CL^L(l[o,(1,0)]))
=\mu(\CL^L([0,1]\times [0,\pi)\times [-L/2,L/2] ))
=\frac{L}{\pi}\int_0^\pi \sin \varphi \dd \varphi=\frac{2L}{\pi}.
\]
Since $\delta>0$ was arbitrary, we see that if 
$\lambda<\frac{\pi}{2}L^{-1},$ we can conclude that $\alpha(\lambda,\CV_\epsilon)<1,$ 
for any $\epsilon>0$ small enough.
By \cite{BJST_09} Lemma 3.5, we conclude that $\CV_\epsilon$ contains geodesics, and 
therefore so does $\CV.$
\end{proof}

We can now prove Proposition \ref{prop:lambdaulower}.
\begin{proof}[Proof of Proposition \ref{prop:lambdaulower}]
First, we note that by the monotonicity explained in Section \ref{sec:stickprocess},
we may assume that $L^{-1}< \lambda<\frac{\pi}{2}L^{-1}$ which we do for convenience.

Next, by Lemma \ref{lemma:CVgeodesics}, $\CV$ contains geodesics with probability one. Therefore, 
there exists two disjoint arcs $I_1,I_2\subset \partial \BH^2$ such that 
with positive probability there is a geodesic in $\CV$ with endpoints in $I_1$ 
and $I_2.$ For a half-plane $H$, let $e(H)$ be the set of boundary points of $H$ which 
belong to $\partial \BH^2.$ Then, let $H_1,H_2$ be two half-planes such that 
$I_1,e(H_1),I_2,e(H_2)$ are all disjoint. Furthermore, choose $H_1,H_2$ so that 
the sets $I_1,e(H_1),I_2,e(H_2) \subset \partial \BH^2$ are situated in that order. 
Placed in this way, a geodesic in $\CV$ separates $H_1$ from $H_2.$

Then, take $L<\infty$
large enough so that $\lambda>L^{-1}>\frac{32\pi}{\sqrt{3}-1}L^{-2}.$ By 
Proposition \ref{prop:half-plane} and Lemma 
\ref{lemma:half-planeperc}, we conclude that both 
$H_1$ and $H_2$ contain an unbounded connected component of $\CC$ with 
probability one. Since these are (with positive probability) separated by the 
geodesic in $\CV,$ we see that the probability
that $\CC$ contains two unbounded components is positive. Hence, $\lambda<\lambda_u(L)$
from which the statement follows.
\end{proof}

\subsection{The upper bound of Theorem \ref{thm:lambdaumain}} 
In this section, we will prove the upper bound of Theorem \ref{thm:lambdaumain}. However, we will 
start by informally explaining the main idea. Here, we write 
\[
A_1 \stackrel{\CV}{\longleftrightarrow} A_2
\]
for the event that $\CV$ 
contains a connected component which intersects both $A_1,A_2\subset \BH^2.$ 
We will prove Lemma \ref{lemma:Vconnect} which shows that
$\BP(B^h(o,1) \stackrel{\CV}{\longleftrightarrow} B^h(x,1))$ is exponentially 
small in the distance $d_h(o,x).$ From this, it will easily follow that the probability of 
$\BP(B^h(o,1) \stackrel{\CV}{\longleftrightarrow} \partial B^h(o,R))$ decays exponentially in $R$
whenever $\lambda \geq 5 \sqrt{2} \pi L^{-1}$. 
We can then show that $\CV$ cannot contain any unbounded connected components separating distinct
connected components of $\CC$, implying uniqueness.

This is the main result of the subsection.
\begin{proposition} \label{prop:lambdauupperbound}
Let $\lambda \geq 5 \sqrt{2} \pi L^{-1}$. Then, for any $0<L<\infty$ large enough
\[
\BP(\CV(\omega^{\lambda,L}) \textrm{ contains an unbounded connected component })=0,
\]
and $\CC$ contains a unique unbounded component. 
Therefore, $\lambda_u(L) \leq 5 \sqrt{2} \pi L^{-1}.$
\end{proposition}
The main step in proving Proposition \ref{prop:lambdauupperbound} will be the following lemma.
\begin{lemma} \label{lemma:Vconnect}
For any $0<L<\infty$ large enough and $\lambda=5 \sqrt{2} \pi L^{-1}$,
we have that
\[
\BP(B^h(o,1) \stackrel{\CV}{\longleftrightarrow} B^h(x,1))
\leq 100e^{-1.15 d_h(o,x)}
\]
for any $x\in \BH^2$.
\end{lemma}
Before proving Lemma \ref{lemma:Vconnect}, let us see how Proposition 
\ref{prop:lambdauupperbound} follows from it.
\begin{proof}[Proof of Proposition \ref{prop:lambdauupperbound} from Lemma \ref{lemma:Vconnect}]
By Lemma \ref{lemma:half-planeperc} and Proposition \ref{prop:half-plane}, 
$\CC$ will contain an unbounded connected 
component with probability 1 whenever $\lambda \geq 5 \sqrt{2} \pi L^{-1}$ 
and $L<\infty$ is large enough. Furthermore, by taking $L$ perhaps even larger, 
we may assume that the conclusion of Lemma \ref{lemma:Vconnect} holds. For any such 
$L,$ we can use monotonicity (see Section \ref{sec:stickprocess}) to conclude 
that it suffices to prove 
the statement for $\lambda=5 \sqrt{2} \pi L^{-1}$ and so  
we also assume that $\lambda=5 \sqrt{2} \pi L^{-1}.$

For any $0<R<\infty,$ we can use Lemma \ref{lemma:ballsoncircle} to conclude that there exists a 
covering of $\partial B^h(o,R)$ by using $N=\lceil \pi \sinh(R) \rceil$ balls of radius 1, 
all centered on $\partial B^h(o,R)$. 
Let $\CB_R$ be the set of those balls $B^h(x_k,1)$ (where $k=1,\ldots,N$) with the 
property that 
\[
B^h(o,1) \stackrel{\CV}{\longleftrightarrow} B^h(x_k,1).
\]
Using Lemma \ref{lemma:Vconnect}, we see that 
\begin{align*}
& \BP(B^h(o,1) \stackrel{\CV}{\longleftrightarrow} \partial B^h(o,R))
=\BP(|\CB_R|>0)  
\leq \BE[|\CB_R|] \\
& \leq \lceil \pi \sinh(R) \rceil 100 e^{-1.15R}
\leq 100 \pi e^R e^{-1.15R} 
\to 0 \textrm{ as } R \to \infty.
\end{align*}
We conclude that the probability that $B^h(o,1)$ touches an unbounded connected component 
of $\CV$ is 0. 
Since we can cover $\BH^2$ by a countable number of balls of radius 1, the first statement 
follows, i.e. that $\CV$ cannot contain an unbounded connected component. 
Next, observe that any two disjoint connected components of 
$\CC$ would have to be separated by an unbounded connected component of $\CV.$
Since by Lemma \ref{lemma:half-planeperc} and Proposition \ref{prop:half-plane}, 
$\CC$ contains an unbounded connected
component, we conclude that this must be unique.
\end{proof}

What remains is to prove Lemma \ref{lemma:Vconnect}, but 
before delving into the proof, we will informally explain the idea behind it. 
Consider therefore the line $l[o,x].$ By Lemma \ref{lemma:altPoisson}, the measure 
of the set of sticks which intersects $l[o,x]$ equals $\frac{2}{\pi}d_h(o,x)L$.
Therefore,
\[
\BP(l[o,x]\subset \CV) \leq e^{-\frac{2\lambda }{\pi}d_h(o,x)L}
\]
coming very close to the statement of Lemma \ref{lemma:Vconnect}. However, even if $l[o,x]$
is ``cut off'' by some $l_L \in \omega^{\lambda,L},$ this does not rule out a connection 
between $B^h(o,1)$ and $B^h(x,1)$ in $\CV,$ as a potential path can go around $l_L$. 
Intuitively however, this should be very unlikely, since such a path would typically 
need to take a rather long detour. Observe also that
if the ends of $l_L$ are connected to the top and bottom parts
of the boundary $\partial \BH^2$ in $\CC,$ this does disconnect $B^h(o,1)$ from $B^h(x,1)$ 
in $\CV.$ 
 Therefore, if a stick lands on $l[o,x],$ this should typically 
suffice in order to break any connection between $B^h(o,1)$ and $B^h(x,1)$ in $\CV.$
Our strategy is then to first find sticks which intersects $l[o,x],$ and then 
use Corollary \ref{cor:seedperc} to show that these indeed separate $B^h(o,1)$ from 
$B^h(x,1)$ in $\CV.$

Let $H^+,H^-$ denote the upper and lower half-planes of $\BH^2$ respectively. That is
\[
H^+:=\{x\in \BH^2: 0<\theta(x)<\pi\} \textrm{ while }
H^-:=\{x\in \BH^2: \pi<\theta(x)<2\pi\}.
\]
Consider the set of sticks $\CL^L([k,k+1]\times [\pi/4,3\pi/4]\times[-L/4,L/4]),$ 
which according to Lemma \ref{lemma:altPoisson} has $\mu$-measure 
\begin{equation} \label{eqn:Lmeasure}
\frac{L}{2 \pi}\int_{\pi/4}^{3\pi/4}\sin \varphi \dd \varphi
=\frac{L}{\sqrt{2}\pi}.    
\end{equation}
Then, fix some
\begin{equation}\label{eqn:lkdef}
l_k=l_k(\rho',r,\varphi)\in \CL^L([k,k+1]\times [\pi/4,3\pi/4]\times[-L/4,L/4])    
\end{equation}
and let $x_k^+\in H^+$ and $x_k^- \in H^-$ be the two points on $l_k$ which are at distance 
$L/8$ from the endpoints of $l_k.$ Furthermore, let $g_k^+$ and $g_k^-$ be the geodesics 
which intersect $x_k^+$ and $x_k^-$ respectively and which are orthogonal to $l_k.$
Then, let $H_k^+$ and $H_k^-$ be the two half-planes defined by $g_k^+$ and $g_k^-$
respectively, and which do not contain $o.$ Finally, let $e_{k,1}^+$ and $e_{k,2}^+$
be the endpoints of $g_k^+$ counted counterclockwise. (We warn the reader that this notation 
clashes somewhat with the notation of Section \ref{sec:lambdac} where $l_0,l_1,l_{-1},l_{1,1}$
etc was used. From here on, the definition \eqref{eqn:lkdef} is what should be considered.) 

\begin{lemma} \label{lemma:Hkdisjoint}
We have that $H_k^+\cap H_{k+4}^+=\emptyset$ and that $H_k^+ \subset H^+$ 
for every $k.$ Similarly, $H_k^-\cap H_{k+4}^-=\emptyset$ and $H_k^- \subset H^-$ 
for every $k.$
\end{lemma}
\begin{proof}
Clearly, it suffices to show the first statement, since the second the follows by symmetry.
Consider therefore some fixed $l_k=l_k(\rho'_k,r_k,\varphi_k)$ as in \eqref{eqn:lkdef}, along with the 
triangle defined by the three points $(\rho'_k,0), x_k^+$ and $e^+_{k,1}$ 
(the situation is 
similar to that of Figure \ref{fig:embeddtree1}).
The angle $\alpha_{k}$ between $l[(\rho'_k,0), x_k^+]$ and 
$l[(\rho'_k,0), e^+_{k,1}]$ 
can, as explained in the proof of Lemma \ref{lemma:half-planes}, be bounded by 
\[
\alpha_k=\sin^{-1}\left(\frac{1}{\cosh(d_h((\rho'_k,0),x_k^+))}\right)
\leq \sin^{-1}\left(\frac{1}{\cosh(L/8)}\right)\leq 2e^{-L/8}.
\]
Since $\varphi_k\in[\pi/4,3\pi/4],$ it follows that $\varphi_k-\alpha_k>0$ and so 
$H_k^+ \subset H^+$ whenever $0<L<\infty$ is large enough.

Next consider the triangle with endpoints $(\rho'_k,0),(\rho'_{k+4},0)$ and 
$e^+_{k+4,2}.$ Then,
let $\beta_k$ be the angle between $l[(\rho'_k,0),(\rho'_{k+4},0)]$ and 
$l[(\rho'_k,0),e^+_{k+4,2}],$ while $\gamma_k$ is the (inner) angle between 
$l[(\rho'_k,0),(\rho'_{k+4},0)]$ and $l[(\rho'_{k+4},0),e_{k+4,2}],$ see Figure
\ref{fig:trianglek}. 
\begin{figure}[h] 
\begin{center}
\includegraphics[page=4, clip=true, trim=85 60 300 40 , width=12cm]{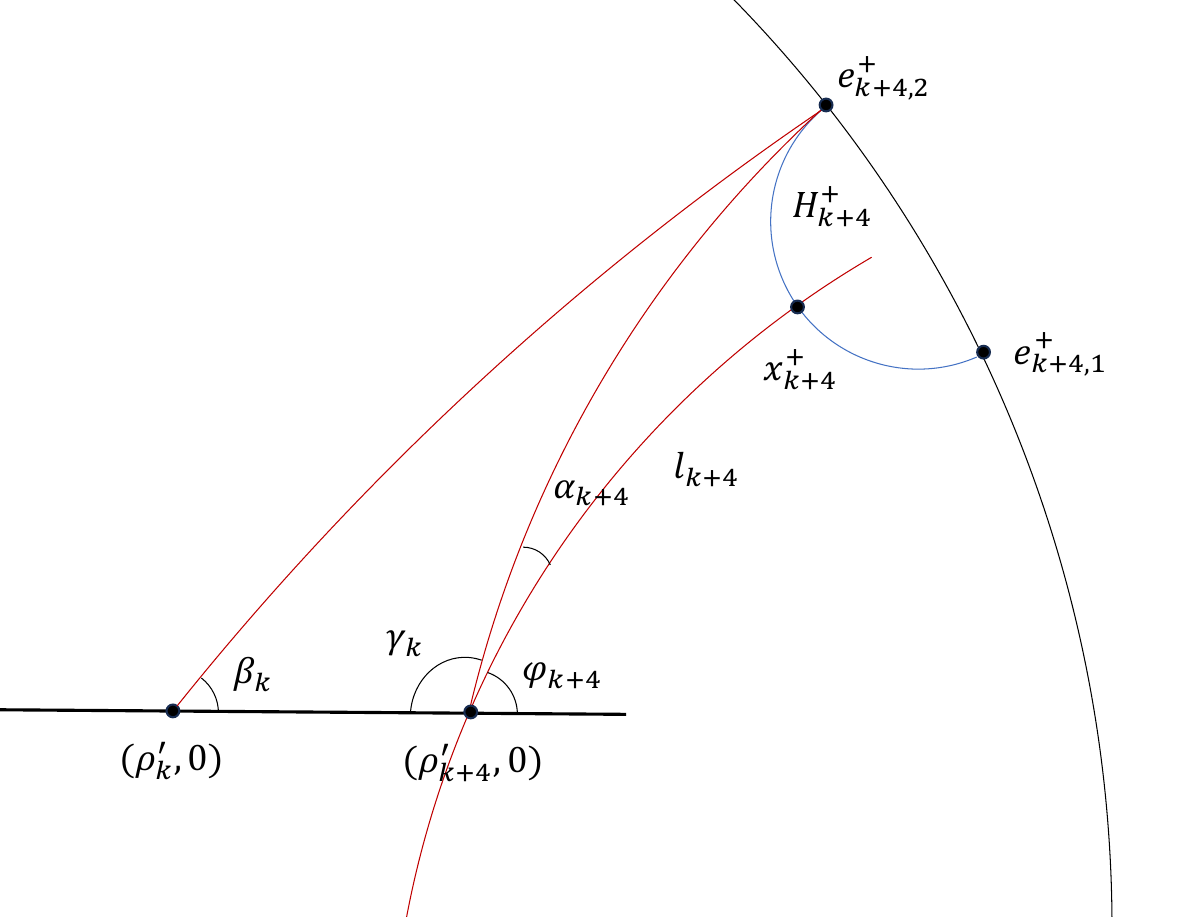} 
\caption{Illustration of angles.} \label{fig:trianglek}   
\end{center}
\end{figure}
Note that since we assume that $\varphi_{k+4}\leq 3 \pi/4$, it follows 
that $\gamma_k= \pi-\varphi_{k+4}-\alpha_{k+4}\geq \pi/4-2e^{-L/8}\geq \pi/5$
for $0<L<\infty$ large enough. Furthermore, 
$\gamma_k \leq \pi-\varphi_{k+4} \leq 3 \pi/4,$ and so 
by the hyperbolic law of cosines \eqref{eqn:hypcosinerule_infinite}, 
we have that 
\[
\cosh(d_h((\rho'_k,0),(\rho'_{k+4},0)))
=\frac{1+\cos(\beta_k)\cos(\gamma_k)}{\sin(\beta_k)\sin(\gamma_k)}
\leq \frac{2}{\sin(\beta_k)\sin(\pi/5)}. 
\]
Then, since $d_h((\rho'_k,0),(\rho'_{k+4},0))\geq 3$, we see that
\[
\beta_k 
\leq \sin^{-1}\left(\frac{2}{\sin(\pi/5)}
\frac{1}{\cosh(3)}\right)\leq \frac{\pi}{9},
\]
where the last inequality can be checked by elementary means. 
We conclude that 
\[
\beta_k+\alpha_k\leq \beta_k+2e^{-L/8}< \pi/8
\]
whenever $0<L<\infty$ is large enough. Then, since $\varphi_k\geq \pi/4,$
it follows that $H_k^+\cap H_{k+4}^+=\emptyset$.
\end{proof}

We are now ready to prove Lemma \ref{lemma:Vconnect}.

\begin{proof}[Proof of Lemma \ref{lemma:Vconnect}]

As in \eqref{eqn:defCH0}, let for any half-plane $H,$
\[
\CC_{|H}:= \bigcup_{l_L\in \omega^{\lambda,L}: l_L\subset H} l_L.  
\]
Then, for any fixed stick $l_L,$ let
$\left(\CC_{|H}\right)_{l_L}\subset \CC_{|H}\cup l_L$ denote the 
connected component of 
$\CC_{|H} \cup l_L$ which contains $l_L.$

According to Corollary \ref{cor:seedperc}, for any $0<p<1$ 
and any fixed $l_k$ as in \eqref{eqn:lkdef},
the probability that 
$\left(\CC_{|H_k^+}\right)_{l_k}$ is an unbounded connected component is larger than $p$ for every 
$0<L<\infty$ large enough. Furthermore, the analog statement holds for $\left(\CC_{|H_k^-}\right)_{l_k}$. 
Let $X_k\in\{0,1\}$ be equal to 1 if there exists 
$l_k\in \omega^{\lambda,L} \cap \CL^L([k,k+1]\times [\pi/4,3\pi/4]\times[-L/4,L/4])$ 
and if $\left(\CC_{|H_k^+}\right)_{l_k}$ and $\left(\CC_{|H_k^-}\right)_{l_k}$ are unbounded.
Clearly, if $X_k=1,$ for some $k=2,3,4,5,\ldots,\lfloor R \rfloor -1$, then the event 
$B^h(o,1) \stackrel{\CV}{\longleftrightarrow} B^h((R,0),1)$ cannot occur. Using Corollary 
\ref{cor:seedperc}, we can therefore conclude that if we let $p=999/1000,$ we have that
for any $0<L<\infty$ large enough
\begin{align*}
& \BP(X_k=1)\geq p^2 \BP\left(\exists l_k\in \omega^{\lambda,L}\cap 
\CL^L([k,k+1]\times [\pi/4,3\pi/4]\times[-L/4,L/4])\right) \\
& =p^2 \left(1-\exp(-\lambda \mu(\CL^L([k,k+1]\times [\pi/4,3\pi/4]\times[-L/4,L/4])))\right) \\
& =p^2 \left(1-\exp\left(- \frac{\lambda L}{\sqrt{2}\pi}\right)\right)
=p^2\left(1-e^{-5}\right)>0.99,
\end{align*}
where we used \eqref{eqn:Lmeasure} and that $\lambda=5\sqrt{2}\pi L^{-1},$ and where the last 
inequality is easily verified by elementary means.

Using Lemma \ref{lemma:Hkdisjoint}, we see that by construction,
$X_k$ and $X_{k+4}$ are independent Bernoulli random variables with parameter 
$0.99,$ and so the sequence $X_2,X_6,X_{10},\ldots$ 
is an i.i.d.\ sequence. Let $M=\lfloor \frac{\lfloor R \rfloor +1}{4}\rfloor$ be the largest integer 
such that $4M-2\leq R-1$, and observe that 
\[
M \geq \frac{\lfloor R \rfloor +1}{4}-1 \geq \frac{R}{4}-1.
\]
We then have that for $R\geq 3$ so that 
$M=\lfloor \frac{\lfloor R \rfloor +1}{4}\rfloor\geq 1,$
\begin{align*}
& \BP(B^h(o,1) \stackrel{\CV}{\longleftrightarrow} B^h((R,0),1)) 
\leq \BP\left(\sum_{m=1}^{M} X_{4m-2}=0\right) 
\\
&=\BP(X_1=0)^{M} 
\leq  \BP(X_1=0)^{R/4-1} 
\leq \exp\left((R/4-1)\log (0.01)\right) \\
& =100\exp\left(\frac{\log(0.01)}{4}R\right)
\leq 100 e^{-1.15 R}.
\end{align*} 
The inequality is trivial for $0\leq R \leq 3$ and so the statement follows.
\end{proof}
\noindent
{\bf Remark:} The situation with Theorem \ref{thm:lambdaumain} is similar to that 
of Theorem \ref{thm:lambdacmain} in that clearly, the value 
$5\sqrt{2}\pi$ in the 
upper bound can be improved. However, in order to make it match the lower bound, 
one would need to come up with new ideas or perhaps improving the existing ones 
to a very large degree. This is currently outside of what we are able to 
do, and so instead of attempting to improve on the constant a little bit, 
we settle for relative simplicity.

\medskip

\noindent
{\bf Acknowledgement:} We would like to thank the anonymous referee for many useful comments.

\end{document}